\documentclass[12pt,a4paper,reqno]{amsart}
\usepackage[activeacute,english]{babel}

\usepackage[applemac]{inputenc}

\usepackage{amsfonts}
\usepackage{amsmath,amsthm,bm}
\usepackage{amssymb}
\usepackage{graphics}
\usepackage{tikz}
\usepackage{float}

\usepackage{hyperref}
\usepackage{cleveref}
\usepackage{autonum}

\DeclareRobustCommand{\SkipTocEntry}[5]{}

\newcommand{\CC}{{\mathcal C}}

\newcommand{\C}{{\mathbb C}}
\newcommand{\N}{{\mathbb N}}
\newcommand{\R}{{\mathbb R}}

\newcommand{\sph}{{\mathbb S}}

\newcommand{\dist}{{\operatorname{dist}}}
\newcommand{\lbop}{{\Delta}}

\newtheorem{theorem}{Theorem}[section]
\newtheorem{lemma}[theorem]{Lemma}
\newtheorem{corollary}[theorem]{Corollary}
\newtheorem{proposition}[theorem]{Proposition}

\newtheorem{remark}[theorem]{Remark}

\newtheorem{definition}[theorem]{Definition}

\numberwithin{equation}{section}

\allowdisplaybreaks[1]

\def\Xint#1{\mathchoice
{\XXint\displaystyle\textstyle{#1}}%
{\XXint\textstyle\scriptstyle{#1}}%
{\XXint\scriptstyle\scriptscriptstyle{#1}}%
{\XXint\scriptscriptstyle\scriptscriptstyle{#1}}%
\!\int}
\def\XXint#1#2#3{{\setbox0=\hbox{$#1{#2#3}{\int}$ }
\vcenter{\hbox{$#2#3$ }}\kern-.6\wd0}}

\def\dashint{\Xint-}

\date{}
\keywords{Minimal energy points, Spherical cap discrepancy, Sobolev discrepancy.}

\begin{document}

\title[Discrepancy of minimal Riesz energy points]
{Discrepancy of minimal Riesz energy points}

\author[J. Marzo]{Jordi Marzo}
\address{J. Marzo \newline
Departament de Matem\`atiques i Inform\`atica,
Barcelona Graduate School of Mathematics (BGSMath),
Universitat de Barcelona,
Gran Via 585,
08007 Barcelona, Spain}
\email{jmarzo@ub.edu}

\author[A. Mas]{Albert Mas}
\address{A. Mas \newline
Departament de Matem\`atiques,
Universitat Polit\`ecnica de Catalunya,
Campus Diagonal Bes\`os, Edifici A (EEBE), Av. Eduard Maristany 16, 08019
Barcelona, Spain}
\email{albert.mas.blesa@upc.edu}

\thanks{The authors are grateful to Joaquim Ortega-Cerd\`a and Carlos Beltr\'an for enlightening discussions on the subject matter of this paper, and to Alexandre Eremenko for information about Wolff's manuscript. 
\\
The first author has been supported by grant MTM2017-83499-P by
the Ministerio de Econom\'ia y Competitividad, Gobierno de Espa\~na and by the Generalitat de
Catalunya (project 2017 SGR 358). The second author was partially supported by
MTM2017-84214 and MTM2017-83499 projects of the MCINN (Spain),
2017-SGR-358 project of the AGAUR (Catalunya), and ERC-2014-ADG project 
HADE Id.\! 669689 (European Research Council).}

\begin{abstract}
We find upper bounds for the spherical cap discrepancy of the set of minimizers of the Riesz $s$-energy
on the sphere $\sph^d.$ Our results are based in bounds for a Sobolev discrepancy 
introduced by Thomas Wolff in an unpublished manuscript where estimates for the spherical cap discrepancy of the logarithmic energy minimizers in $\sph^2$ were obtained. 
Our result improves previously known bounds 
for $0\le s<2$ and $s\neq 1$ in $\sph^2,$ where $s=0$ is Wolff's result, and for 
$d-t_0<s<d$ with $t_0\approx 2.5$ when $d\ge 3$ and $s\neq d-1.$
\end{abstract}

\maketitle

\section{Introduction and main results}

For an $N$ point set $X_N=\{x_1,\dots ,x_N\}$ in the unit sphere $\sph^d=\{x\in\R^{d+1}:\,|x|=1\}$ and $0\le s<d$ the 
Riesz $s$-energy of $X_N$ is given by
\begin{equation} \label{s_energy}
E_s(X_N)=\sum_{i\neq j}\frac{1}{ |x_i-x_j|^s},\quad\text{if }0<s<d,
\end{equation}
and
\begin{equation}	\label{log_energy}
E_0(X_N)=\sum_{i\neq j}\log \frac{1}{ |x_i-x_j|},\quad\text{if }s=0.
\end{equation}
We denote the minimal 
Riesz $s$-energy, for $0\le s<d,$ achived by an $N$ point set 
by
\begin{equation} \label{min energy discrete}
\mathcal{E}_s(N)=\inf_{X_N} E_s(X_N),
\end{equation}
where $X_N$ 
runs through the $N$ point sets $X_N\subset \sph^d.$

Problems related to these minimal energies or with the minimizers, in the spherical and in other settings, have been
extensively studied. It is well known that the continuous Riesz $s$-energy of the surface measure on the sphere
gives the leading term of the asymptotic expansion of the normalized discrete energy.
Moreover, due to the work of different authors, see \cite{BHS12,BG15} and references therein, it is known that for $d\ge 2$ and $0<s<d$ 
there exist constants $C,\,c>0$ such that
\begin{equation}		\label{knownboundsenergy}
-c N^{1+s/d}\le \mathcal{E}_s(N)-E_s(\widetilde{\sigma}) N^2 
\le -C N^{1+s/d} 
\end{equation}
for $N\ge 2,$ where 
\begin{equation}  		\label{continuousrieszenergy}
E_s(\widetilde{\sigma})=\int_{\sph^d}\! \int_{\sph^d} \frac{1}{ |x-y |^s}\, 
d\widetilde{\sigma}(x)\,d\widetilde{\sigma}(y)=2^{d-s-1}
\frac{\Gamma\left( \frac{d+1}{2} \right)\Gamma\left( \frac{d-s}{2} 
\right)}{\sqrt{\pi}\Gamma\left( d-\frac{s}{2} \right)}
 \end{equation}
and $\widetilde{\sigma}$ is the normalized surface measure in $\sph^d$ given by the relation
$\sigma=\omega_d \widetilde{\sigma}$, being $\sigma$ the surface measure in the sphere and
$\omega_d=\sigma(\sph^d)=2 \pi^{\frac{d+1}{2}}/ \Gamma\left(\frac{d+1}{2} \right).$ 
For the logarithmic case $s=0$ it is known that
\begin{equation}			\label{knownboundsenergy2}
\mathcal{E}_0(N)=E_0(\widetilde{\sigma}) N^2 -\frac{1}{d}N\log N+O(N), 
\end{equation}
where \[
E_0(\widetilde{\sigma})= \int_{\sph^d}\! \int_{\sph^d} \log\frac{1}{|x-y|} 
\,d\widetilde{\sigma}(x)\,d\widetilde{\sigma}(y)= \frac{\psi_0(d)-\psi_0(d/2)}{2}-\log 2
\]
and $\psi_0$ denotes the digamma function.

Several conjectures about the lower order terms on these asymptotic expansions and, in some particular dimensions, about the value of the constants appearing 
in the asymptotic expansion can be found in the literature, see \cite{BMOC16,Bet14,BHS12,BG15,Ser15}.

There are still many basic open questions about the distribution of the minimizers. 
Recall that $X_N=\{x_1,\dots ,x_N\}\subset \sph^d$ is separated if  
$$\min_{i\neq j}|x_i-x_j|\ge C N^{-\frac{1}{d}},$$
for some constant $C>0,$  \cite{BG15}. The minimizers are known to be separated if $0\le s<2$ for $d=2$ and $d-2 \le s <d$ for $d\ge 3,$
but 
in $0\le s <d-2$ and $d\ge 3,$ 
the best bound is $O(N^{-\frac{1}{s+2}}),$
\cite{Dah78,DM05, DragnevSaff,KSS07}.

It is classical 
that, for any $0\le s<d,$ minimizers of the Riesz $s$-energy in $\sph^d$ are 
asymptotically uniformly distributed, meaning that
$$\lim_{N\to +\infty} \frac{1}{N}\sum_{j=1}^N f(x_j)=\int_{\sph^d}f(x)\,d\widetilde{\sigma}(x),\quad\mbox{for all } f\in \mathcal{C}(\sph^d),$$
or, equivalently, that the sum of delta measures $\mu_N=\frac{1}{N}\sum_{j=1}^N \delta_{x_j} $ converges in the weak-$*$ topology to the normalized surface measure 
$\widetilde{\sigma}.$ 
It is also a well known fact that the $N$ point sets $\{X_N\}_{N}$ are asymptotically uniformly distributed if and only if the spherical cap discrepancy
converges to zero
\[
\lim_{N\to +\infty} \sup_{x\in \sph^d,r>0} \Bigl| \frac{\# (X_N\cap D_r(x))}{N}- 
\widetilde{\sigma}(D_r(x)) \Bigr|=0,
\]
where $D_r(x)=\{y\in\sph^d:|x-y|<r\}$ is a spherical cap of center $x\in\sph^d$ and (euclidean) distance $r>0.$ 
Loosely speaking, the speed of this convergence is a measure of how well distributed are the points in $\{X_N\}_N .$ 

Our objective is to provide 
upper bounds for the spherical cap discrepancy of $N$ point sets of minimizers of the Riesz $s$-energy. Previous work around this
problem focused basically in the Coulomb potential case $s=d-1.$ This case is somehow simpler because, for $d\ge 2$, $|x|^{1-d}$ is (modulo a constant) the fundamental solution for the Laplacian in $\R^{d+1}$  
$$-\lbop (|x|^{1-d})=(d-1)\omega_d \delta_0,$$
where $\delta_0$ is a Dirac delta at the origin. Points which minimize the $(d-1)$-energy are called Fekete points and, in this 
setting, the results are tipically valid for regular enough $d-$dimensional surfaces in $\R^{d+1},$ not only for $\sph^d.$
The first result is due to Kleiner \cite{Klei64} and yields that the spherical cap discrepancy of a set of Fekete points on the 
sphere is $O(N^{-\frac{1}{3d}}).$ 
Sj\"ogren \cite{Sjo73} improved this result to $O(N^{-\frac{1}{2d}}).$
In 1996 Korevaar \cite{Kor96} conjectured that the right bound was $O(N^{-\frac{1}{d}})$.
Finally, G\"otz \cite{Got00} proved Korevaar's conjecture, up to a logarithmic factor, giving the best known result $O(N^{-\frac{1}{d}}\log N)$.

For all other cases, $s\neq d-1,$ the only known results are due to Brauchart, who established the bound
\begin{equation} \label{Brauchart_bound}
 O(N^{-\frac{d-s}{d(d-s+2)}}),\;\;\mbox{if}\;\;0\le s<d,
\end{equation}
see \cite{Bra05} for $0<s<d$ and \cite{Bra08} for the logarithmic case.
Observe that, in the harmonic case $s=d-1$, Brauchart's result gives a bound of the same order as Kleiner's.

However, for the logarithmic case in $\sph^2,$ Wolff proved in an unpublished manuscript \cite{Wolff} the bound 
$O(N^{-1/3})$,
which is better than the $O(N^{-1/4})$ following from Brauchart (\ref{Brauchart_bound}). For more information about Wolff's manuscript see Remark \ref{remark_historico}. Our objective in this work is to generalize 
 Wolff's approach to other $s$-energies in $\sph^d.$ 
In this regard, our main result is an upper bound for the spherical cap discrepancy of the 
energy minimizers that improves Brauchart's result in the range $0\le s<2$ for $d=2$
(where $s=0$ is Wolff's result) and in the range 
$d-t_0<s<d,$ for $d\ge 3,$ where
$t_0=\frac{1+\sqrt{17}}{2}>2.$ 
In all these cases, when $s=d-1$ G\"otz's mentioned result is still the best one.

\begin{theorem}				\label{teo:main}
Let $0\le s<d$ and $X_N=\{ x_1,\dots ,x_N \}$ be an $N$-point set of minimizers of the Riesz $s$-energy in $\sph^d$. Then
\begin{equation}
\sup_D \Bigl| \frac{\# (X_N\cap D)}{N}- 
\widetilde{\sigma}(D) \Bigr|
\lesssim 
\chi_{[0,d-2]}(s) N^{-\frac{2}{d(d-s+1)}}
+\chi_{(d-2,d)}(s)N^{-\frac{2(d-s)}{d(d-s+4)}}
\end{equation}
with constants depending only on $d$ and $s$, where the supremum on the left hand side runs over all spherical caps $D\subset \sph^d.$ 
\end{theorem}

\begin{remark}{\em 
The same bound above holds when the discrepancy is defined in terms of the so called $K$-regular sets 
instead of the spherical caps, \cite{Sjo73}. 
}\end{remark}


%

Observe that, in the harmonic case $s=d-1,$ our result gives a bound between Kleiner's and Sj\"ogren's results. Note also that
all these results, ours and G\"otz's, are far from the optimal spherical cap discrepancy established by Beck for $N$ point 
sets in $\sph^d,$ which is of order $N^{-\frac{d+1}{2d}},$ up to a logarithmic term, \cite{Bec84}. 

Theorem \ref{teo:main} gives a quantitative proof of the asymptotic equidistribution of the energy minimizers. It would 
be interesting 
to extend our approach to Green energies on manifods like the ones studied in \cite{BCCdR19}.

Following Wolff's approach, Theorem \ref{teo:main} on the spherical cap discrepancy will follow from a sharp estimate 
of a discrepancy defined in terms of Sobolev norms. We will introduce now the needed concepts.


\subsection{Spherical harmonics and Sobolev discrepancy}

Given an integer $\ell\ge 0$, let $\mathcal{H}_\ell$ be the vector space of 
the spherical harmonics of degree $\ell$, i.e. the space of eigenfunctions
    of the the Laplace-Beltrami operator $\Delta$ on $\sph^{d}$, 
\begin{equation}\label{ev laplace}
-\Delta Y=\ell (\ell +d-1)Y,\;\;\;\; Y\in \mathcal{H}_{\ell},
\end{equation}
of eigenvalue $\ell (\ell +d-1).$ The value
$h_\ell=\operatorname{dim}\mathcal{H}_\ell,$ 
is the multiplicity of the eigenvalue $\ell (\ell +d-1)$ and
it is easily seen to be $h_\ell\approx \ell^{d-1}.$

For the Hilbert space $L^2(\sph^d)$ of square integrable functions in $\sph^d$
with the inner product
\[
\langle f,g \rangle=\int_{\sph^d} f(x) g(x)\,d\sigma(x),\;\;\; f,\,g\in L^2(\sph^d),
\]
one has the decomposition $L^2(\sph^d)=\bigoplus_{\ell\ge 0} \mathcal{H}_\ell.$ Therefore, for $f\in L^2(\sph^d)$, one has the Fourier representation
\begin{equation}\label{fourier representation f}
f=\sum_{\ell,k} f_{\ell,k} Y_{\ell,k},\qquad f_{\ell,k}=\langle f,Y_{\ell,k} \rangle=\int_{\sph^d} f\, Y_{\ell,k} \,d\sigma,
\end{equation}
where $\{ Y_{\ell,k} \}_{k=1}^{h_\ell}$ is an orthonormal basis of $\mathcal{H}_\ell.$

Given $r\ge 0$ we consider the standard $L^2(\sph^d)$-based Sobolev spaces of order $r$ defined in terms of their representation on the 
Fourier side, namely, 
$$\mathbb{H}^{r}(\sph^d)=\left\{ f\in L^2(\sph^d)\;\;:\;\; 
\sum_{\ell=0}^{+\infty}\sum_{k=1}^{h_\ell} (1+\ell^2)^{r} |f_{\ell,k}|^2<+\infty \right\},$$
with the norm
$$\| f \|_{\mathbb{H}^{r}(\sph^d)}=\left( \sum_{\ell=0}^{+\infty}\sum_{k=1}^{h_\ell} (1+\ell^2)^{r} |f_{\ell,k}|^2  \right)^{1/2}.$$ 
Since $(1+\ell^2)^{r}$, $(1+\ell^r)^{2}$, and  $1+\ell^{2r}$ are comparable for all $\ell\geq0$ with constants only 
depending on $r$, in the sequel we may use at our convenience any of these expressions to 
estimate the norm $\| \cdot \|_{\mathbb{H}^{r}(\sph^d)}$. Recall that $\mathbb{H}^{r}(\sph^d)$ is 
continuously embedded in $\mathcal{C}^k(\sph^d)$ if $ r-k >d/2$. 

For any Borel measure $\mu$ in $\sph^d$,
we consider a ``dual'' Sobolev norm of $\mu$ defined by
$$\| \mu \|_{\mathbb{H}^{-r}(\sph^d)}=\sup \left\{ \int_{\sph^d} \psi\,  d\mu  \;:\;\psi\in \mathcal{C}^{\infty}(\sph^d),\;\; \| \psi \|_{\mathbb{H}^{r}(\sph^d)}=1  \right\}.$$
When the measure is of the form $\mu=h\sigma$ for some $h\in L^2(\sph^d)$, by an abuse of notation we will simply write $\| h \|_{\mathbb{H}^{-r}(\sph^d)}$.

Following \cite{LPS86,Wolff}, we define the following discrepancy with respect to functions in the Sobolev space.

\begin{definition}\label{defi sobolev discrep mu}
  Let $X_N=\{x_1,\dots , x_N \}$ be a set of $N$ points on the sphere $\sph^d.$ Given $\epsilon>0$ and $0\le s<d,$ the Sobolev 
discrepancy of $X_N$ is
\begin{equation}
D_{s,d}^\epsilon(X_N)=  \| \mu_{X_N,\epsilon} \|_{\mathbb{H}^{\frac{s-d}{2}}(\sph^d)},
\end{equation}
where 
$$\mu_{X_N,\epsilon}= \Big(\frac{1}{N} \sum_{j=1}^N \frac{\chi_{D_j}}{\sigma(D_j)}-\frac{1}{\omega_d} \Big)\sigma,$$
and $D_j=D_{\epsilon N^{-1/d}}(x_j).$
\end{definition}

\begin{remark}\label{remark_1}
{\em 
In \cite{Wolff} Wolff considered a homogenous Sobolev norm instead. 
But both in the original work of Wolff and in the present article, these norms are used to pass from the spherical cap discrepancy 
(an $L^\infty$ estimate) to the Sobolev discrepancy (a ``dual'' Sobolev estimate), and then to the asymptotics of the energy. 
One can check in the proof of Theorem \ref{teo:main} below that
the zero order term, say $\|f\|_{L^2(\sph^d)}$, can be absorbed by the dominant term and our final conclusion completely agrees with the one of Wolff.
}\end{remark}

Our following result is an estimate of the Sobolev 
discrepancy of minimizers that is sharp in the range $d-2\le s<d.$

\begin{theorem}				\label{teo_sobolev_discrepancy}
Let $0\le s<d$ and $X_N=\{ x_1,\dots ,x_N \}$ be an $N$ point set of minimizers of the Riesz $s$-energy in $\sph^d$. Then, for every $\epsilon>0$ small enough depending only on $d$ and $s$, 
$$N^{-\frac{1}{2}+\frac{s}{2d}}\lesssim D_{s,d}^\epsilon(X_N)\lesssim N^{-\frac{1}{d}}+N^{-\frac{1}{2}+\frac{s}{2d}}$$
with constants depending only $d$, $s$, and $\epsilon$.
\end{theorem}

\begin{remark}{\em 
It is well known that the linearization of the quadratic Wasserstein distance is precisely the homogenous Sobolev norm considered by Wolff (see Remark \ref{remark_1}) \cite[7.6]{Vil03}. 
Moreover, Peyre has recently shown that the quadratic Wasserstein distance is bounded above by the the homogenous Sobolev norm and therefore Wolff's result can be read in terms of 
a quadratic Wasserstein discrepancy, \cite{Pey18}.  
}\end{remark}


\subsection{Notation}

Given $d\geq2$ integer, we denote by $\lbop$ and $\nabla$ the spherical Laplacian and spherical gradient on $\sph^d\subset\R^{d+1}$, respectively. 
Given $s\geq0$ and two different points 
$x,\,y\in\R^{d+1}$, let the Riesz kernel of order $s$ acting on $(x,y)$ be defined by
\begin{equation}
R_s(x,y):=
\begin{cases}
|x-y|^{-s},&\quad\text{for $s>0$,}\\
-\log|x-y|,&\quad\text{for $s=0$.}
\end{cases}
\end{equation}
For $0\leq s< d$ and $f\in L^2(\sph^d)$, we define the spherical Riesz transform of 
$f$ by
\begin{equation}
R_s f(x):=\int_{\sph^d}  R_s(x,y) f(y)\,d\sigma(y),\qquad\text{for } x\in\sph^d.
\end{equation}
We denote the Riesz $s$-energy of a Borel measure $\mu$ in $\sph^d$ by
\begin{equation}\label{riesz energy definition measure}
E_s(\mu):=\int_{\sph^d}\!\int_{\sph^d}  R_s(x,y) \,d\mu(y)\,d\mu(x).
\end{equation}
When $\mu=f\sigma$ with $f\in L^2(\sph^d)$, we write $E_s(f)$ instead of $E_s(f\sigma)$, hence
\begin{equation}\label{riesz energy definition function}
E_s(f)=\int_{\sph^d} f(x)\,R_sf(x)\,d\sigma(x)=\int_{\sph^d}\!\int_{\sph^d}  R_s(x,y) f(x)f(y)\,d\sigma(y)\,d\sigma(x).
\end{equation}

\subsection{Structure of the article}
Section \ref{s2} contains the preliminaries. In there, we set the basic properties of the Riesz kernels and the spherical Riesz transform, namely, we show that the operator $R_s$ diagonalizes in the standard basis of spherical harmonics, we find its eigenvalues in a closed form and their asymptotic behavior, and we study the relation between the Riesz kernels and the Laplace-Beltrami operator on the sphere, giving some heuristics.

Section \ref{s ass Riesz energies} focuses on asymptotic estimates of Riesz energies on the sphere. The main result is an estimate of the continuous Riesz energy of small discs centered at the discrete minimizers in terms of the minimal energy $\mathcal{E}_s(N)$. This, together with the asymptotic expansion of the minimal energy, is a key tool to derive the estimates of the Sobolev discrepancy given in Theorem \ref{teo_sobolev_discrepancy}, which is proven in Section \ref{s4}.

Finally, in Section \ref{s5} we give the proof of Theorem \ref{teo:main}, which is a straightforward application of  Theorem \ref{teo_sobolev_discrepancy} and Proposition \ref{sobolev_SCdiscrepancy}.


\section{Spectral analysis of the Spherical Riesz transform}\label{s2}

In this section we set basic properties of the Riesz kernels and the spherical Riesz transform. On one hand, we show that the operator $R_s$ diagonalizes in the standard basis of spherical harmonics. In addition, we find its eigenvalues in a closed form using hypergeometric functions, and we analyze their asymptotic behavior. This is the purpose of Proposition \ref{spherical ham diag}, a key result that will be systematically used in the sequel. In particular, it allows us to relate in Lemma \ref{lemma_equivalence} below the Riesz energies $E_s$ to the dual Sobolev norms $\| \cdot \|_{\mathbb{H}^{(s-d)/2}(\sph^d)}$ for $0\leq s<d$.

On the other hand, we explore the relation between the Riesz kernels and the Laplace-Beltrami operator on the sphere. Essentially, we show that the kernel 
$R_{s+2}$ can be obtained by applying the Helmoltz differential operator $-\Delta+C_{d,s}$ to the kernel $R_s$, where $C_{d,s}$ is a suitable constant depending on $d$ and $s$; see Lemma \ref{laplacian_of_riesz} for the precise statement. In the particular case of $s=d-2>0$, we get that $R_{d-2}$ is a multiple of a fundamental solution of $-\Delta+(d-2)d/4$. These identities are the departing point in Section \ref{s ass Riesz energies} to get the asymptotic estimates for Riesz energies.

Even though the asymptotic behavior of the eigenvalues for the spherical Riesz transform is obtained in 
Proposition \ref{spherical ham diag} for the whole range $0\leq s<d$ by a direct argument, it is of interest to see how one 
can get it in the subcritical regime ($0< s<d-2$) from its knowledge in the critical ($s=d-2$) and supercritical ($d-2<s<d$) 
regimes by an iteration argument based on the connection between $R_{s}$ and $R_{s+2}$ mentioned before. We develop this argument 
at the end of this section. This served us to see how, for every positive integer $m$, the Sobolev 
norms $\| \cdot \|_{\mathbb{H}^{m}(\sph^d)}$ defined in terms of the spherical harmonics decomposition correspond 
to the standard Sobolev norms given by pure derivatives, and gave us an intuition for extending Wolff's arguments for $(s,d)=(0,2)$ to the whole range $0\leq s<d$. These last considerations are treated in Lemma \ref{expressions in pure derivatives}.


\subsection{Fourier multipliers}

This part is devoted to show that the spherical harmonics diagonalize the spherical Riesz transform, and to find the asymptotic behavior of the eigenvalues. With this at hand, we find a simple expression in terms of the Fourier coefficients which serves to connect the Riesz energy to a dual Sobolev norm.

For the expression of the Riesz potential and the Riesz energy of a function $f\in L^2(\sph^d)$ written in terms of spherical harmonics, we recall the following definition of a generalized hypergeometric function.

\begin{definition}
For integers $p,\,q\ge 0$ and complex values $a_i,\,b_j$,
the generalized hypergeometric function is defined by the power series
\begin{equation}	\label{eq:qFp}
\,{}_pF_q(a_1,\ldots,a_p;b_1,\ldots,b_q;z) = \sum_{n=0}^\infty 
\frac{(a_1)_n\dots(a_p)_n}{(b_1)_n\dots(b_q)_n} \, \frac {z^n} {n!},
 \end{equation}
where $(\cdot)_n$ is the rising factorial or Pochhammer symbol given by $(x)_0=1$  and
\[
(x)_n=x(x+1)\cdots(x+n-2)(x+n-1)=\frac{\Gamma(x+n)}{\Gamma(x)},\quad n\geq1,
\] 
for $x\in\C$.
\end{definition}

\begin{proposition}\label{spherical ham diag}
Let $0\leq s<d$ and let $\{ Y_{\ell,k} \}_{\ell,k}$ for $\ell=0,1,\ldots$ and $k=1,\ldots h_\ell$ be an orthonormal basis of 
spherical harmonics in $L^2(\sph^d).$ 
Given $f\in L^2(\sph^d)$, we have
\begin{equation}\label{symbol Riesz1}
R_sf(x) =\int_{\sph^d}  R_s(x,y) f(y)\,d\sigma(y)=\sum_{\ell,k}A_{\ell,s}\,f_{\ell,k}\,Y_{\ell,k}(x)
\end{equation}
for almost all $x\in\sph^d$, and
\begin{equation}\label{symbol Riesz1bis}
E_s(f)
=\int_{\sph^d}\!\int_{\sph^d}  R_s(x,y) f(x)f(y)\,d\sigma(y)\,d\sigma(x)=\sum_{\ell,k}A_{\ell,s}\,|f_{\ell,k}|^2,
\end{equation}
where $f_{\ell,k}=\int_{\sph^d} f\,Y_{\ell,k}\,d\sigma$ and 
\begin{equation}\label{symbol Riesz2}
\begin{split}
A_{\ell,s} & =\frac{2^{d-s}\Gamma\left( \frac{d-s}{2} \right)}{\Gamma\left( d-\frac{s}{2} \right)}\;
{}_3 F_2 \Big(\!\!-\ell,\ell+d-1,\frac{d-s}{2};\frac{d}{2},d-\frac{s}{2};1\Big)
\\
&
=
\frac{2^{d-s}  \Gamma\left( \frac{d-s}{2} \right) \Gamma\left( \frac{s}{2}+\ell \right)  }{\Gamma\left(\frac{s}{2} \right) \Gamma\left(d-\frac{s}{2}+\ell \right) }.
\end{split}
\end{equation}
Additionally, there exists $C>0$ only depending on $s$ and $d$ such that
\begin{equation}\label{symbol Riesz3}
\frac{C^{-1}}{1+\ell^{d-s}}\leq A_{\ell,s}\leq\frac{C}{1+\ell^{d-s}}\qquad\text{for all }\ell\geq0.
\end{equation}
In particular,
\begin{equation}\label{symbol Riesz4}\begin{split}
E_s(f)\approx\sum_{\ell,k}\frac{1}{1+\ell^{d-s}}\,|f_{\ell,k}|^2.
\end{split}
\end{equation}
\end{proposition}

\proof
We first consider the case $s>0$. If we set $F_s(t)=(2-2t)^{-s/2}$, then $R_s(x,y)=F_s(\langle x, y \rangle).$
By Funk-Hecke formula, see \cite[page 11]{DX13}, 
\begin{equation}
\begin{split}
\int_{\sph^d} F_s(\langle x, y \rangle) Y_{\ell,k}(x)\,d\sigma(x)=
\frac{\omega_{d-1} }{ C^{\frac{d-1}{2}}_\ell(1)}\left(\int_{-1}^1 F_s(t)C^{\frac{d-1}{2}}_\ell(t)(1-t^2)^{\frac{d-2}{2}} dt \right)   Y_{\ell,k}(y),
\end{split} 
\end{equation}
and the expression for $R_sf(x)$ follows if we set
$$A_{\ell,s}=\frac{\omega_{d-1} }{2^{s/2} C^{\frac{d-1}{2}}_\ell(1)}
 \int_{-1}^1 C^{\frac{d-1}{2}}_\ell(t)(1-t )^{\frac{d-s}{2}-1}(1+t )^{\frac{d-2}{2}} dt .$$
The expression for $E_s(f)$ follows then by orthogonality, that is,
$$\int_{\sph^d}\!\int_{\sph^d}  R_s(x,y) Y_{\ell,k}(x) Y_{\ell',k'}(y) \,d\sigma(x)\,d\sigma(y)=A_{\ell,s}\delta_{(\ell,k),(\ell',k')}.$$

From \cite[page 281]{EOT54} one can get a closed expression in terms of an hypergeometric function
\begin{equation}			\label{integral:closed:form}
 \begin{split}
&  \int_{-1}^1   C^{\frac{d-1}{2}}_\ell(t)(1-t )^{\frac{d-s}{2}-1}(1+t )^{\frac{d-2}{2}} dt 
\\
&\quad
=2^{d-\frac{s}{2}-1}
\frac{\Gamma\left( \frac{d-s}{2}\right) \Gamma\left(\frac{d}{2} \right) \Gamma\left(\ell+d-1 \right)  }{\Gamma\left(\ell+1 \right) 
\Gamma\left(d-1 \right) \Gamma\left(d-\frac{s}{2} \right) }
{}_3 F_2 \Big(\!\!-\ell,\ell+d-1,\frac{d-s}{2};\frac{d}{2},d-\frac{s}{2};1\Big).
 \end{split}
\end{equation}
From Saalsh\"utz's theorem we get
 $${}_3 F_2 \Big(\!\!-\ell,\ell+d-1,\frac{d-s}{2};\frac{d}{2},d-\frac{s}{2};1\Big)=
\frac{\Gamma\left( \frac{s}{2}+\ell \right) \Gamma\left(d-\frac{s}{2} \right) }{\Gamma\left(\frac{s}{2} \right) \Gamma\left(d-\frac{s}{2}+\ell \right) },$$
which yields \eqref{symbol Riesz2}.

Finally, the asymptotic expression for the quotient of gamma functions
$$\lim_{n\to +\infty} \frac{\Gamma(n+\alpha)}{\Gamma(n)n^\alpha}=1,\quad\alpha\in \C,$$
proves (\ref{symbol Riesz3}) because, clearly, $A_{\ell,s}\neq 0.$ 

The endpoint case $s=0$ is obtained using $F_0(t)=-\frac{1}{2}\log (2-2t)$ and taking the derivative with respect to $s$ and evaluating at $s=0$
the expression (\ref{integral:closed:form}).
\qed

\begin{remark}{\em
From Proposition \ref{spherical ham diag} it follows that, at a formal level,
\begin{equation}
 \begin{split}
 \frac{1}{|x-y|^s} & =\sum_{\ell,k}A_{\ell,s}Y_{\ell,k}(x)Y_{\ell,k}(y)= 
 \sum_{\ell}A_{\ell,s}  \sum_{k}Y_{\ell,k}(x)Y_{\ell,k}(y)
 \\
 &
 =\sum_{\ell}\frac{A_{\ell,s}}{\omega_d} \frac{2\ell +d-1}{d-1}\,C_\ell^{\frac{d-1}{2}}(\langle x, y\rangle),
 \end{split}
\end{equation}
where $\langle x, y\rangle$ is the cosine of the angle between $x$ and $y$, and 
$C_\ell^{\alpha}(t)$ is the Gegenbauer polynomial orthogonal in $[-1,1]$ with respect to $(1-t^2)^{\alpha-\frac{1}{2}}$ with the normalization  
$C_\ell^\alpha (1)=\binom{2\alpha+k-1}{k}$.
But as 
$$\frac{A_{\ell,s}}{\omega_d}=\frac{2^{d-s-1} \Gamma\left( \frac{d+1}{2} \right) \Gamma\left( \frac{d-s}{2} \right) \Gamma\left( \frac{s}{2}+\ell \right)  }{\sqrt{\pi}\Gamma\left(\frac{s}{2} \right) \Gamma\left(d-\frac{s}{2}+\ell \right) },$$
we get that 
$A_{\ell,d-1}=\omega_d (d-1)/(2\ell +d-1)$ and, thus,
\begin{equation}
 \frac{1}{|x-y|^{d-1}}
 =\sum_{\ell} C_\ell^{\frac{d-1}{2}}(\langle x, y\rangle).
\end{equation}
In particular, it is well known that for the Newtonian potential and $x,\,y\in \sph^2$  one has
\begin{equation}
  \frac{1}{|x-y|}=
\sum_{\ell} P_\ell (\langle x, y\rangle),
\end{equation}
where $P_\ell (t)$ is the Legendre polynomial of degree $\ell$ normalized as $P_\ell(1)=1.$
}\end{remark}

Using Proposition \ref{spherical ham diag}, in the following result we highlight the important connection between 
$\| \cdot \|_{\mathbb{H}^{(s-d)/2}(\sph^d)}$ and the Riesz $s$-energy $E_s$ introduced in \eqref{riesz energy definition measure}. 
\begin{lemma}					\label{lemma_equivalence}
Given $0\leq s<d$, there exists a constant $C>0$ only depending on $s$ and $d$ such that, for every $h\in L^2(\sph^d)$,
\begin{equation}\label{relation sobolev riesz}
C^{-1}\| h \|^2_{\mathbb{H}^{(s-d)/2}(\sph^d)}\leq E_s(h)
\leq C\| h \|^2_{\mathbb{H}^{(s-d)/2}(\sph^d)}.
\end{equation}
\end{lemma}
\begin{proof}
By a limiting argument, it suffices to prove the lemma when $h$ is smooth. On one hand, in \eqref{symbol Riesz4} we showed that if $h=\sum_{\ell,k} h_{\ell,k} Y_{\ell,k}$ then 
\begin{equation}
E_s(h)\approx\sum_{\ell,k}\frac{1}{1+\ell^{d-s}}\,|h_{\ell,k}|^2.
\end{equation}
Writing every $\psi\in \mathcal{C}^{\infty}(\sph^d)$ as 
$\psi=\sum_{\ell,k} \psi_{\ell,k} Y_{\ell,k}$, by Cauchy-Schwarz inequality we have
\begin{equation}
\begin{split}
\Big|\int_{\sph^d} \psi h\,d\sigma\Big|^2
&=\Big|\sum_{\ell,k}h_{\ell,k}\psi_{\ell,k}\Big|^2
\leq\Big(\sum_{\ell,k}\frac{|h_{\ell,k}|^2}{1+\ell^{d-s}}\Big)
\Big(\sum_{\ell,k}({1+\ell^{d-s}})|\psi_{\ell,k}|^2\Big)\\
&\approx E_s(h)\| \psi \|^2_{\mathbb{H}^{(d-s)/2}(\sph^d)},
\end{split}
\end{equation}
which yields the first inequality in \eqref{relation sobolev riesz} by taking the supremum on $\psi\in \mathcal{C}^{\infty}(\sph^d)$.

Let us now prove the second inequality. Since $h$ is smooth by assumption, it is not hard to see that indeed $R_sh$ too. 
Thanks to \eqref{symbol Riesz1} and \eqref{symbol Riesz3} we have that
$R_sh=\sum_{\ell,k}A_{\ell,s}\,h_{\ell,k}\,Y_{\ell,k}$ with 
$A_{\ell,s}\approx(1+\ell^{d-s})^{-1}$. Hence, we can estimate
\begin{equation}
\begin{split}
E_s(h)&=\int_{\sph^d} h\, R_sh\,d\sigma
\leq\| h \|_{\mathbb{H}^{(s-d)/2}(\sph^d)}\| R_sh \|_{\mathbb{H}^{(d-s)/2}(\sph^d)}\\
&\approx\| h \|_{\mathbb{H}^{(s-d)/2}(\sph^d)}\Big( \sum_{\ell,k} (1+\ell^{d-s}) 
|A_{\ell,s}\,h_{\ell,k}|^2  \Big)^{1/2}
\approx\| h \|_{\mathbb{H}^{(s-d)/2}(\sph^d)}E_s(h)^{1/2},
\end{split}
\end{equation} 
and the second inequality in \eqref{relation sobolev riesz} follows.
\end{proof}


\subsection{The Laplace-Beltrami operator on the Riesz kernel}
In this section we give useful identities which connect the Riesz kernels of different indexes through the Laplace-Beltrami operator. All of them are collected in the following lemma, which will be systematically used in Section \ref{s ass Riesz energies} to get the asymptotic estimates for Riesz energies. 

\begin{lemma}								\label{laplacian_of_riesz}
Let $x_0\in\sph^{d}$. Then for $d\ge 2$ and $s>0,$ as a function of $x\in\sph^{d}$,
\begin{equation}\label{lbop through laplace eq1}
\Big(\!\!-\lbop + \frac{1}{4}\,s(2d-2-s)\Big)R_s(x,x_0)=s(d-2-s)R_{s+2}(x,x_0),
\quad \text{for all $x\neq x_0$.}
\end{equation}
Furthermore, in the special case of $d>2$ and $s=d-2$ we have
\begin{equation}\label{lbop through laplace eq2}
\Big(\!\!-\lbop + \frac{1}{4}\,(d-2)d\Big)R_{d-2}(\cdot,x_0)
=C_d (d-2)\delta_{x_0},
\end{equation}
in the sense of distributions, 
where $\delta_{x_0}$ denotes the Dirac measure at $x_0$ and 
$C_d=2{\pi^{d/2}}/{\Gamma\left(d/2\right)}.$ Indeed,
for every  open set $\Omega\subset\sph^d$ with smooth boundary and such that $x_0\in\Omega$, and every
$f\in \mathcal{C}^2(\Omega)\cap \mathcal{C}^1(\overline{\Omega})$, we have
\begin{equation}\label{lbop through laplace eq3}
\begin{split}
C_d (d-2)f(x_0)&=
\int_{\Omega}R_{d-2}(x,x_0)
\,\Big(\!\!-\lbop + \frac{1}{4}\,(d-2)d\Big)f(x)\,d\sigma(x)\\
&\quad+\int_{\partial\Omega}\!\! 
\big(R_{d-2}(x,x_0)\,\nabla f(x)-f(x)\,\nabla R_{d-2}(x,x_0)\big)
\!\cdot\! \nu(x)\,d\sigma'(x),
\end{split}
\end{equation}
where $\sigma'$ denotes the $(d-1)-$dimensional Hausdorff measure.

In the logarithmic case $s=0$ and $d>2$, we have
\begin{equation}			\label{log laplace eq1}			
-\Delta R_0 (x,x_0)=(d-2) R_2 (x,x_0)-\frac{d-1}{2},
\quad \text{for all $x\neq x_0,$}
\end{equation}
and, when $s=0$ and $d=2$,
\begin{equation}				\label{log laplace eq2}
-\Delta R_0(\cdot,x_0)=2\pi \delta_{x_0}-\frac{1}{2}
\end{equation}
in the sense of distributions. That is, for every open set $\Omega\subset\sph^2$ with smooth boundary and such that $x_0\in\Omega$, and every
$f\in \mathcal{C}^2(\Omega)\cap \mathcal{C}^1(\overline{\Omega})$, we have

\begin{equation}				\label{log laplace eq3}
\begin{split}
2\pi f(x_0)&=
-\int_{\Omega}R_{0}(x,x_0)
\Delta f(x)\,d\sigma(x) +\frac{1}{2}\int_\Omega f(x)\, d\sigma(x)   \\
&\quad+\int_{\partial\Omega}\!\! 
\big(R_{0}(x,x_0)\,\nabla f(x)-f(x)\,\nabla R_{0}(x,x_0)\big)
\!\cdot\! \nu(x)\,d\sigma'(x).
\end{split}
\end{equation}
\end{lemma}


\begin{proof}
Let $\Delta_{\R^{d+1}}$ denote the standard Laplacian in $\R^{d+1}$ and 
$\pi:\R^{d+1}\setminus\{0\}\to\sph^d$ be the spherical projection given by 
$\pi(y)=y/|y|$ for $y\in\R^{d+1}\setminus\{0\}$. 
It is well known that if $f:\sph^{d}\to\R$ then the spherical Laplacian of $f$ can be computed through $\Delta_{\R^{d+1}}$ by the formula
\begin{equation}\label{lbop through laplace}
\lbop f(x)=(\Delta_{\R^{d+1}}(f\circ\pi))(\pi(y))
\end{equation}
at the points $x=\pi(y)$ where $f\circ\pi$ is twice differentiable.

We will consider the case $s>0,$ the logarithmic case $s=0$ follows easily along the same lines.
Take $x_0\in\sph^{d}$. A computation shows that, 
as a function of $y\in\R^{d+1}\setminus\{0\}$, 
\begin{equation}
\begin{split}
\Delta_{\R^{d+1}} \Big(R_s&(\pi(\cdot),x_0)\Big)(y)\\
&=s\Big\{(s+2)R_{s+4}(\pi(y),x_0)f
\Big(\frac{|x_0|^2}{|y|^2}-\frac{(y\cdot x_0)^2}{|y|^4}\Big)
-dR_{s+2}(\pi(y),x_0)\frac{(y\cdot x_0)}{|y|^3}\Big\}.
\end{split}
\end{equation}
Using the definition of $R_s$, that $x=\pi(y)$ and that 
$|x_0|=|x|=1$, we then get
\begin{equation}
\begin{split}
\Delta_{\R^{d+1}} \Big(R_s&(\pi(\cdot),x_0)\Big)\big(\pi(y)\big)\\
&=s\Big\{(s+2)R_{s+4}(x,x_0)
\Big(1-(x\cdot x_0)^2\Big)
-dR_{s+2}(x,x_0)(x\cdot x_0)\Big\}\\
&=s|x-x_0|^{-s-2}\Big\{(s+2)\frac{1-(x\cdot x_0)^2}{|x-x_0|^{2}}
-d(x\cdot x_0)\Big\}.
\end{split}
\end{equation}
Observe that 
\begin{equation}\label{lbop through laplace eq2 interm 4}
1-(x\cdot x_0)^2=\big(1+(x\cdot x_0)\big)\big(1-(x\cdot x_0)\big)
=\big(1+(x\cdot x_0)\big)\frac{|x-x_0|^2}{2}.
\end{equation}
Therefore,
\begin{equation}
\begin{split}
\Delta_{\R^{d+1}} \Big(R_s(\pi(\cdot),x_0)\Big)&\big(\pi(y)\big)
=s|x-x_0|^{-s-2}\Big\{\frac{s+2}{2}\big(1+(x\cdot x_0)\big)
-d(x\cdot x_0)\Big\}\\
&=\frac{s}{2}|x-x_0|^{-s-2}\Big\{(2d-2-s)\big(1-(x\cdot x_0)\big)+2(s+2-d)\Big\}\\
&=\frac{s}{2}|x-x_0|^{-s-2}\Big\{(2d-2-s)\frac{|x-x_0|^2}{2}+2(s+2-d)\Big\}\\
&=\frac{s}{4}(2d-2-s)R_s(x,x_0)+s(s+2-d)R_{s+2}(x,x_0)
\end{split}
\end{equation}
which, together with \eqref{lbop through laplace}, yields 
\eqref{lbop through laplace eq1}. For the logarithmic case $s=0$, the previous computations lead to 
\begin{equation}
\begin{split}
\Delta_{\R^{d+1}} \Big(R_0(\pi(\cdot),x_0)\Big)\big(\pi(y)\big)
&=\frac{1}{2}|x-x_0|^{-2}\Big\{(2d-2)\frac{|x-x_0|^2}{2}+2(2-d)\Big\}\\
&=\frac{d-1}{2}+(2-d)R_{2}(x,x_0),
\end{split}
\end{equation}
which proves \eqref{log laplace eq1}.

We now address \eqref{lbop through laplace eq2}. Given $x_0\in\Omega\subset\sph^d$ and $\epsilon>0$ set 
$\Omega_\epsilon:=\Omega\setminus D_\epsilon(x_0)$, whose boundary is the disjoint union of $\partial\Omega$ and $\partial D_\epsilon(x_0)$ if $\epsilon$ is small enough. We denote by $\nu$ the outward unit normal vector (tangent to $\sph^d$) on $\partial\Omega_\epsilon$.
An integration by parts gives
\begin{equation}\label{lbop through laplace eq2 interm}
\begin{split}
\int_{\Omega_\epsilon}&R_{d-2}(x,x_0)
\,\Big(\!\!-\lbop + \frac{1}{4}\,(d-2)d\Big)f(x)\,d\sigma(x)\\
&=\int_{\Omega_\epsilon}\!\!\nabla R_{d-2}(x,x_0)\!\cdot\!\nabla f(x)\,d\sigma(x)
-\int_{\partial\Omega_\epsilon}\!\! R_{d-2}(x,x_0)\,\nabla f(x)
\!\cdot\!\nu(x)\,d\sigma'(x)\\
&\quad+\frac{1}{4}\,(d-2)d\int_{\Omega_\epsilon}R_{d-2}(x,x_0)f(x)\,d\sigma(x)\\
&=\int_{\partial\Omega_\epsilon}\!\! \nabla R_{d-2}(x,x_0)
\!\cdot\! \nu(x)\, f(x)\,d\sigma'(x)
-\int_{\partial\Omega_\epsilon}\!\! R_{d-2}(x,x_0)\,\nabla f(x)
\!\cdot\! \nu(x)\,d\sigma'(x)\\
&\quad+\int_{\Omega_\epsilon}
\Big(\!\!-\lbop + \frac{1}{4}\,(d-2)d\Big)R_{d-2}(x,x_0)\, f(x)\,d\sigma(x).
\end{split}
\end{equation}
Since $\dist(x_0,\Omega_\epsilon)>0$, the last term on the right hand side of 
\eqref{lbop through laplace eq2 interm} vanishes by \eqref{lbop through laplace eq1}. Note also that
\begin{equation}\label{lbop through laplace eq2 interm 2}
\begin{split}
\Big|\int_{\partial D_\epsilon(x_0)}\!\! R_{d-2}(x,x_0)\,\nabla f(x)&
\!\cdot\! \nu(x)\,d\sigma'(x)\Big|\\
&\leq \|\nabla f\|_{L^\infty(\Omega)}\epsilon^{-d+2}
\sigma(\partial D_\epsilon(x_0))=O(\epsilon).
\end{split}
\end{equation}
Arguing as in \eqref{lbop through laplace}, we have that 
$\nabla R_{d-2}(x,x_0)=-(d-2)R_{d}(x,x_0)((x\cdot x_0)x-x_0).$
Moreover,
for $x\in\partial D_\epsilon(x_0)$ the outward unit normal vector with respect to 
$\Omega_\epsilon\subset\sph^d$ is 
\begin{equation}\label{lbop through laplace eq2 interm 5}
\nu(x)=\frac{x_0-x-((x_0-x)\cdot x)x}{|x_0-x\cdot x-((x_0-x)\cdot x)x|}
=\frac{x_0-(x_0\cdot x)x}{|x_0-(x_0\cdot x)x|},
\end{equation}
which leads to
$\nabla R_{d-2}(x,x_0)\cdot\nu(x)
=(d-2)R_{d}(x,x_0)|x_0-(x_0\cdot x)x|$. Observe also that,
from \eqref{lbop through laplace eq2 interm 4}, 
$|x_0-(x_0\cdot x)x|^2=1-(x_0\cdot x)^2
=\frac{1}{2}\big(1+(x_0\cdot x)\big)|x-x_0|^2.$
Therefore,
\begin{equation}\label{lbop through laplace eq2 interm 6bis}
\begin{split}
\int_{\partial D_\epsilon(x_0)}\!\! \nabla R_{d-2}(x,x_0)
\!\cdot\! \nu(x)\, &f(x)\,d\sigma'(x)\\
&=\frac{d-2}{\epsilon^{d-1}}\int_{\partial D_\epsilon(x_0)}\!\!
\Big(\frac{1+(x_0\cdot x)}{2}\Big)^{1/2}f(x)\,d\sigma'(x).
\end{split}
\end{equation}
Since the integrand is continuous near $x_0$, we deduce that
\begin{equation}\label{lbop through laplace eq2 interm 3}
\lim_{\epsilon\to0}\int_{\partial D_\epsilon(x_0)}\!\! \nabla R_{d-2}(x,x_0)
\!\cdot\! \nu(x)\, f(x)\,d\sigma'(x)=C_d (d-2)f(x_0),
\end{equation}
where 
$$C_d:=\lim_{\epsilon\to0}\epsilon^{1-d}
\sigma'(\partial D_\epsilon(x_0))=\frac{2\pi^{d/2}}{\Gamma\left( \frac{d}{2}\right)}.$$
Finally, taking the limit $\epsilon\to0$ in \eqref{lbop through laplace eq2 interm} and using 
\eqref{lbop through laplace eq2 interm 2} and 
\eqref{lbop through laplace eq2 interm 3} we get 
\eqref{lbop through laplace eq3}. The statement in \eqref{lbop through laplace eq2} is a consequence of \eqref{lbop through laplace eq3} taking $\Omega=\sph^d$, thus $\partial\Omega=\emptyset$.

Observe that in the logarithmic case $\nabla R_{0}(x,x_0)=-R_{2}(x,x_0)((x\cdot x_0)x-x_0),$ and then
formula (\ref{lbop through laplace eq2 interm 6bis}) becomes
\begin{equation}\label{lbop through laplace eq2 interm 6bis bis}
\begin{split}
\int_{\partial D_\epsilon(x_0)}\!\! \nabla R_{0}(x,x_0)
\!\cdot\! \nu(x)\, f(x)\,d\sigma'(x)
&=\frac{1}{\epsilon}\int_{\partial D_\epsilon(x_0)}\!\!
\Big(\frac{1+(x_0\cdot x)}{2}\Big)^{1/2}f(x)\,d\sigma'(x).
\end{split}
\end{equation}
Hence, the same argument yields the case $d=2.$
\end{proof}



\subsection{From the supercritical to the subcritical regime through iteration} \label{ss super to sub}
By a direct argument, in Proposition \ref{spherical ham diag} we found the asymptotic behavior of the eigenvalues for the Spherical Riesz transform for the whole range $0\leq s<d$, see \eqref{symbol Riesz3}. However, it is of interest to see how one can get it in the subcritical regime from its knowledge in the critical and supercritical regimes by an iteration argument based on \eqref{lbop through laplace eq1}. This is the purpose of this section.

We begin by showing, directly from \eqref{lbop through laplace eq2}, the asymptotics \eqref{symbol Riesz3} in the critical regime  $0<s=d-2$.
Given $f\in L^2(\sph^d)$ 
set $f_{\ell,k}=\int f\,Y_{\ell,k}\,d\sigma$, hence \eqref{symbol Riesz1} gives
$R_sf=\sum_{\ell,k}A_{\ell,s}\,f_{\ell,k}\,Y_{\ell,k}.$
Thanks to \eqref{lbop through laplace eq2} and \eqref{ev laplace} we get
\begin{equation}\label{symbol Riesz frac to all eq1}
\begin{split}
\sum_{\ell,k}f_{\ell,k}\,Y_{\ell,k}&=f=\frac{1}{C_d (d-2)}
\Big(\!\!-\lbop + \frac{1}{4}\,(d-2)d\Big)R_{d-2}f\\
&=\frac{1}{C_d (d-2)}\Big(\!\!-\lbop + \frac{1}{4}\,(d-2)d\Big)
\sum_{\ell,k}A_{\ell,d-2}\,f_{\ell,k}\,Y_{\ell,k}\\
&=\sum_{\ell,k}f_{\ell,k}\,\frac{A_{\ell,d-2}}{C_d }
\Big(\frac{\ell(\ell+d-1)}{d-2} + \frac{d}{4}\Big)Y_{\ell,k}.
\end{split}
\end{equation}
Since this holds for all $f\in L^2(\sph^d)$ we deduce that 
\begin{equation}
A_{\ell,d-2}=C_d 
\Big(\frac{\ell(\ell+d-1)}{d-2} + \frac{d}{4}\Big)^{-1}
\end{equation}
and \eqref{symbol Riesz3} follows in this case.

Assuming now that \eqref{symbol Riesz3} holds in the (super)critical regime $0<d-2\leq s<d$, let us deal with the case $s\in(0, d-2)$. Let $m\in\N$ be such that $s\in[d-2(m+1),d-2m)$, thus indeed $m$ is the unique integer such that 
$(d-s)/2-1\leq m<(d-s)/2$. Then $s+2m\in[d-2,d)$ and, by assumption, 
\begin{equation}\label{symbol Riesz frac to all eq2}
A_{\ell,s+2m}\approx\frac{1}{1+\ell^{d-s-2m}}\qquad\text{for all }\ell\geq0.
\end{equation}
Furthermore, if we set $s_j=s+2j$ for $j=0,1,2,\ldots,m$,
iterating \eqref{lbop through laplace eq1} we deduce that
\begin{equation}\label{lbop through laplace eq1 aux111}
\begin{split}
R_{s_m}(\cdot,x)&=\frac{-\lbop + s_{m-1}(\frac{d-1}{2}-\frac{s_{m-1}}{4})}
{s_{m-1}(d-2-s_{m-1})}\,R_{s_{m-1}}(\cdot,x)\\
&=\frac{-\lbop + s_{m-1}(\frac{d-1}{2}-\frac{s_{m-1}}{4})}
{s_{m-1}(d-2-s_{m-1})}\cdots
\frac{-\lbop + s_{0}(\frac{d-1}{2}-\frac{s_{0}}{4})}
{s_{0}(d-2-s_{0})}\,R_{s_{0}}(\cdot,x)
\end{split}
\end{equation}
Then, similarly to what we did in \eqref{symbol Riesz frac to all eq1}, 
from \eqref{symbol Riesz1}, \eqref{lbop through laplace eq1 aux111}, and \eqref{ev laplace} we have
\begin{equation}
\begin{split}
\sum_{\ell,k}A_{\ell,s+2m}&\,f_{\ell,k}\,Y_{\ell,k}
=R_{s_m}f\\
&=\frac{-\lbop + s_{m-1}(\frac{d-1}{2}-\frac{s_{m-1}}{4})}
{s_{m-1}(d-2-s_{m-1})}\cdots
\frac{-\lbop + s_{0}(\frac{d-1}{2}-\frac{s_{0}}{4})}
{s_{0}(d-2-s_{0})}\,R_{s_{0}}f\\
&=\frac{-\lbop + s_{m-1}(\frac{d-1}{2}-\frac{s_{m-1}}{4})}
{s_{m-1}(d-2-s_{m-1})}\cdots
\frac{-\lbop + s_{0}(\frac{d-1}{2}-\frac{s_{0}}{4})}
{s_{0}(d-2-s_{0})}
\sum_{\ell,k}A_{\ell,s}\,f_{\ell,k}\,Y_{\ell,k}\\
&=\sum_{\ell,k}A_{\ell,s}\frac{\ell(\ell+d-1) + s_{m-1}(\frac{d-1}{2}-\frac{s_{m-1}}{4})}
{s_{m-1}(d-2-s_{m-1})}\\
&\hskip100pt\cdots\frac{\ell(\ell+d-1) + s_{0}(\frac{d-1}{2}-\frac{s_{0}}{4})}
{s_{0}(d-2-s_{0})}\,f_{\ell,k}\,Y_{\ell,k}.
\end{split}
\end{equation}
This combined to \eqref{symbol Riesz frac to all eq2} leads to
\begin{equation}
\begin{split}
A_{\ell,s}&=A_{\ell,s+2m}
\frac{s_{m-1}(d-2-s_{m-1})}{\ell(\ell+d-1) + s_{m-1}(\frac{d-1}{2}-\frac{s_{m-1}}{4})}
\cdots\frac{s_{0}(d-2-s_{0})}{\ell(\ell+d-1) + s_{0}(\frac{d-1}{2}-\frac{s_{0}}{4})}\\
&\approx\frac{1}{(1+\ell^{d-s-2m})(1+\ell^{2m})}\approx\frac{1}{1+\ell^{d-s}}
\end{split}
\end{equation}
for all $\ell\geq0$, and \eqref{symbol Riesz3} follows in the subcritical regime.

\begin{remark}{\em
From the previous computations, if $d>2$ we can get the explicit expressions 
\begin{equation}
A_{\ell,d-2}=\frac{2\pi^{d/2}}{\Gamma\left(d/2\right)}
\cdot\frac{1}{\frac{\ell(\ell+d-1)}{d-2} + \frac{d}{4}}
\end{equation}
and
\begin{equation}
\begin{split}
A_{\ell,d-2k}=A_{\ell,d-2}\,
\frac{2}{\frac{\ell(\ell+d-1)}{d-4} + \frac{d+2}{4}}\cdot
\frac{4}{\frac{\ell(\ell+d-1)}{d-6} + \frac{d+4}{4}}
\cdots\frac{2(k-1)}{\frac{\ell(\ell+d-1)}{d-2k} + \frac{d+2(k-1)}{4}}
\end{split}
\end{equation}
for all $k\geq 2$ integer such that $d-2k>0$. A similar formula can be shown for $A_{\ell,0}$ if $d=4,6,8,\ldots$ by taking into account \eqref{log laplace eq1} to pass from $A_{\ell,0}$ to $A_{\ell,2}$. We omit the details.
}\end{remark}


\subsection{The connection to Sobolev spaces on the sphere}\label{ss spectral.4}

The computations in Section \ref{ss super to sub} served us to see how, for every positive integer $m$, the Sobolev norms $\| \cdot \|_{\mathbb{H}^{m}(\sph^d)}$ defined in terms of the spherical harmonics decomposition correspond to the standard Sobolev norms given by pure derivatives, giving us an intuition for extending Wolff's arguments for the case $(s,d)=(0,2)$ to the whole range $0\leq s<d$. In order to clarify this, let us first make some considerations on the Sobolev spaces $\mathbb{H}^{m}(\sph^d)$.
Of course, $\mathbb{H}^{0}(\sph^d)=L^2(\sph^d)$. Looking at the spherical harmonics, for every given $j\in\N$
an integration by parts and \eqref{ev laplace} show that
\begin{equation}\label{from space to fourier harmonics}
\begin{split}
\int_{\sph^d}  |\Delta^jY_{\ell,k}|^2\,d\sigma
=\int_{\sph^d}  Y_{\ell,k}\,\,\Delta^{2j} Y_{\ell,k}\,d\sigma
=\ell^{2j} (\ell +d-1)^{2j}\approx\ell^{4j}
\end{split}
\end{equation}
and
\begin{equation}\label{from space to fourier harmonics2}
\begin{split}
\int_{\sph^d}  |\nabla \Delta^jY_{\ell,k}|^2\,d\sigma
=-\int_{\sph^d}  Y_{\ell,k}\,\,\Delta^{2j+1} Y_{\ell,k}\,d\sigma
=\ell^{2j+1} (\ell +d-1)^{2j+1}\approx\ell^{4j+2}
\end{split}
\end{equation}
for all $\ell\geq0$, with constants only depending on $d$ and $j$. Moreover, by the orthogonality of the basis $\{ Y_{\ell,k} \}_{k=1, \ell\geq0}^{h_\ell}$, we also get
\begin{equation}\label{from space to fourier harmonics3}
\begin{split}
\int_{\sph^d}  Y_{\ell,k}\,\,\Delta^{j} Y_{\ell',k'}\,d\sigma=0
\end{split}
\end{equation}
for all $j\in\N\cup\{0\}$ whenever $(\ell,k)\neq(\ell',k')$.

Given an odd number $m\in\N$ set $i=(m-1)/2$. Then,
using \eqref{fourier representation f}, \eqref{from space to fourier harmonics}, \eqref{from space to fourier harmonics2}, and \eqref{from space to fourier harmonics3}, for every $f\in L^2(\sph^d)$ we see that
\begin{equation}
\begin{split}
\| f \|_{\mathbb{H}^{m}(\sph^d)}^2
&\approx\sum_{\ell=0}^{+\infty}\sum_{k=1}^{h_\ell} (1+\ell^{2m})|f_{\ell,k}|^2
\approx\sum_{\ell=0}^{+\infty}\sum_{k=1}^{h_\ell}\sum_{j=0}^{i} 
(\ell^{4j}+\ell^{4j+2})|f_{\ell,k}|^2\\
&\approx\sum_{j=0}^{i}\sum_{\ell=0}^{+\infty}\sum_{k=1}^{h_\ell} 
\Big(\int_{\sph^d}  Y_{\ell,k}\,\,\Delta^{2j} Y_{\ell,k}\,d\sigma
-\int_{\sph^d}  Y_{\ell,k}\,\,\Delta^{2j+1} Y_{\ell,k}\,d\sigma\Big)|f_{\ell,k}|^2\\
&=\sum_{j=0}^{i}\Big(\int_{\sph^d}\!\!  f\,\Delta^{2j}\! f\,d\sigma
-\!\int_{\sph^d} \!\! f\,\Delta^{2j+1}\! f\,d\sigma\Big)
=\sum_{j=0}^{i}\int_{\sph^d}  \big(|\Delta^{j} f|^2+|\nabla\Delta^{j} f|^2\big)\,d\sigma,
\end{split}
\end{equation}
where the comparability constants only depend on $d$ and $m$. The same argument applies in case that $m\in\N$ is even. Thus, we get the following well-known result.
\begin{lemma}\label{expressions in pure derivatives}
For $m=0,2,4,6,8,\ldots$ we have
\begin{equation}
\begin{split}
\| f \|_{\mathbb{H}^{m}(\sph^d)}^2
&\approx\|f\|^2_{L^2(\sph^d)}+\|\nabla f\|^2_{L^2(\sph^d)}
+\|\Delta f\|^2_{L^2(\sph^d)}+\|\nabla\Delta f\|^2_{L^2(\sph^d)}\\
&\quad+\ldots+\|\Delta^{\frac{m-2}{2}} f\|^2_{L^2(\sph^d)}+
\|\nabla\Delta^{\frac{m-2}{2}} f\|^2_{L^2(\sph^d)}
+\|\Delta^{\frac{m}{2}} f\|^2_{L^2(\sph^d)},
\end{split}
\end{equation}
and for $m=1,3,5,7,9,\ldots$ we have
\begin{equation}
\begin{split}
\| f \|_{\mathbb{H}^{m}(\sph^d)}^2
&\approx\|f\|^2_{L^2(\sph^d)}+\|\nabla f\|^2_{L^2(\sph^d)}
+\|\Delta f\|^2_{L^2(\sph^d)}+\|\nabla\Delta f\|^2_{L^2(\sph^d)}\\
&\quad+\ldots+\|\nabla\Delta^{\frac{m-3}{2}} f\|^2_{L^2(\sph^d)}+\|\Delta^{\frac{m-1}{2}} f\|^2_{L^2(\sph^d)}+\|\nabla\Delta^{\frac{m-1}{2}} f\|^2_{L^2(\sph^d)}.
\end{split}
\end{equation}
\end{lemma}
With the expressions of $\| \cdot\|_{\mathbb{H}^{m}(\sph^d)}$ from 
Lemma \ref{expressions in pure derivatives} at hand, one can now take a new look to \eqref{relation sobolev riesz}.


\section{Asymptotic estimates of Riesz energies} \label{s ass Riesz energies}

This section focuses on asymptotic estimates of Riesz energies on the sphere. The main result, namely Corollary \ref{corollary_bounds_on_means}, 
is an estimate of the continuous Riesz energy of small discs centered at the discrete minimizers in terms of the minimal energy $\mathcal{E}_s(N)$ defined in \eqref{min energy discrete}, plus error terms. This, together with the asymptotic expansion of the minimal energy, will be a key tool in the next section to derive the estimates of the Sobolev discrepancy given in Theorem \ref{teo_sobolev_discrepancy}.

To prove Corollary \ref{corollary_bounds_on_means}, we treat the supercritical and subcritical regimes separately. In the first one, we essentially make use of the separation of the point minimizers, a decomposition of the sphere in dyadic annuli, and Gauss-Green formula \eqref{lbop through laplace eq3}. This is carried out in Lemma \ref{lm 32}, Proposition \ref{prop lbop through laplace main estimate ds1}, and Theorem \ref{teo_estimate_difference} below. Since in the subcritical case the separation property is not known to hold, in Lemma \ref {lemma_mean} below we overcome this difficulty by making use of the fact that $R_{s}(\cdot, x_0)$ is superharmonic near
$x_0$ when $0\le s<d-2$ for $d>2$, as \eqref{lbop through laplace eq1} shows. We mention that the original argument of Wolff for the logarithmic kernel in $\sph^2$ was already based on the use of superharmonicity.


\begin{lemma}				\label{lemma_mean}
For $d>2$ and $0< s<d-2$ there exist $\delta,\,C>0$ depending only on $s$ and $d,$ such that for every $a,b\in \sph^d$,
\begin{equation}
 \dashint_{D_r(a)}\dashint_{D_r(b)} R_s(x,y) \,d\sigma(x)\,d\sigma(y)\le R_s(a,b)+C r^2,
\end{equation}
with, say, $0<r<\frac{\delta}{100}.$ 

For $d\ge 2$ there exist $C>0$ depending only on $d$ such that for every $a,b\in \sph^d$,
\begin{equation}
\dashint_{D_r(a)}\dashint_{D_r(b)}  R_0 (x,y) \,d\sigma(x) \le R_0 (a,b)+C r^2,
\end{equation}
for all $0<r\le 1.$
\end{lemma}


\begin{proof}
Let $a,\,b\in \sph^d$ and $\delta>0$ small enough to be chosen later on. We take $r>0$ such that, say, $0<r<\frac{\delta}{100}$ and we split the argument into two cases,
$$\;\; D_r (b)\subset D_{2\delta}(a)\;\; \mbox{or}\;\; \;\; D_r (b)\subset \sph^d\setminus D_{\delta}(a).$$ 
Observe that from Lemma \ref{laplacian_of_riesz} lemma we get 
$$\Delta_x R_s(x,y)= s\left[ \frac{(2d-s-2)}{4}|x-y|^2-(d-s-2)  \right]R_{s+2}(x,y).$$
Thus, there exist $\delta>0$ small enough depending on $s$ and $d$ such that for $|x-y|<10\delta$ 
$$\Delta_x R_s(x,y)\le 0.$$

In the first case $D_r (b)\subset D_{2\delta}(a),$
we get $D_r(a),\,D_r(b)\subset D_{2\delta}(a)$ and, therefore,
$$\Delta_x R_s(x,y)\le 0,\qquad \Delta_y R_s(x,y)\le 0,$$
for all $x\in D_r(a)$ and $y\in D_r(b).$
Then we get
$$\dashint_{D_r(b)} R_s(a,y) d\sigma(y)\le R_s(a,b)\quad\mbox{and} \quad\dashint_{D_r(a)} R_s(x,y) d\sigma(x)\le R_s(a,y),$$
for all $y\in D_r(b)$ and, therefore,
\begin{equation}    \label{case1}
\dashint_{D_r(a)}\dashint_{D_r(b)} R_s(x,y) \,d\sigma(x)\,d\sigma(y)\le R_s(a,b).
\end{equation}

In the second case $D_r (b)\subset \sph^d\setminus D_{\delta}(a),$ we take the $\delta>0$ as in the first case. Observe that for 
$x\in D_r(a)$ and $y\in D_r(b)$ we have $|x-y|>\frac{\delta}{2}$ and therefore, from the explicit expressions above,
$\Delta_x R_s(x,y)$ is bounded above by a constant $C_{s,d}>0$ depending only on $s$ and $d.$ By a computation similar to one in Lemma \ref{laplacian_of_riesz} we have
$$\Delta_x |x-x_0|^2=d(2-|x-x_0|^2)$$
for all $x_0\in \sph^d$,
and therefore if $|x-x_0|\le 1$ we have $\Delta_x |x-x_0|^2\ge d.$ Then, taking $C=C_{s,d}/d$ and $x_0=a$, we get
$$\Delta_x (R_s(x,y)-C |x-a|^2)\le 0,\quad x\in D_r(a).$$
From this superharmonicity and the corresponding mean value inequality we obtain
$$\dashint_{D_r(a)} R_s(x,y)d\sigma(x)- C \dashint_{D_r(a)} |x-a|^2 d\sigma(x)\le R_s(a,y)$$
and, thus,
\begin{equation}    \label{uno}
\dashint_{D_r(a)} R_s(x,y)d\sigma(x)\le R_s(a,y)+ C r^2 .
\end{equation}
Similarly, we get
$$\dashint_{D_r(b)} R_s(a,y)d\sigma(y)- C \dashint_{D_r(b)} |y-b|^2 d\sigma(y)\le R_s(a,b),$$
and
\begin{equation}    \label{dos}
\dashint_{D_r(b)} R_s(a,y)d\sigma(y)\le R_s(a,b)+ C r^2 .
\end{equation}
Combining (\ref{uno}) and (\ref{dos}) we finally obtain 
$$\dashint_{D_r(a)}\dashint_{D_r(b)} R_s(x,y) \,d\sigma(x)\,d\sigma(y)\le R_s(a,b)+C r^2,$$
and the result follows together with (\ref{case1}).

In the logarithmic case we argue as above and take $a,\,b\in \sph^d$ and $0<r\le 1.$ Then, in the distributional sense, for every $y \in D_r(b)$,
$$\Delta_x \Big(R_0(\cdot ,y )-\frac{d-1}{2d} |\cdot -a|^2\Big)\le 0,\quad\mbox{in}\;\;D_r(a),$$
and
$$\Delta_y \Big(R_0(a,\cdot )-\frac{d-1}{2d} |\cdot -b|^2\Big)\le 0,\quad\mbox{in}\;\;D_r(b).$$
It follows that
$$\dashint_{D_r(a)}\dashint_{D_r(b)} R_0 (x,y) \,d\sigma(x)\,d\sigma(y)\le R_0 (a,b)+\frac{d-1}{d} r^2.$$
\end{proof}



\begin{lemma}\label{lm 32}
Let $d>2$, $0< s<d$, $r_0>0$, and $x_0,\,x_1\in\sph^{d}$ be such that $|x_0-x_1|>r_0$.
Then,
\begin{equation}\label{lbop through laplace form1}
\begin{split}
R_s(x_0,x_1)
&=\frac{ds(d-2-s)}{(d-2)C_d r_0^{d}}
\int_0^{r_0}\!\int_{D_r(x_0)}\!\!r^{d-1}
\big(R_{d-2}(x,x_0)-r^{2-d}\big)R_{s+2}(x,x_1)\,d\sigma(x) \,dr\\
&\quad+\frac{d(d-s)(d-2-s)}{4(d-2)C_d r_0^{d}}
\int_0^{r_0}\!\int_{D_r(x_0)}r^{d-1}R_{d-2}(x,x_0)R_s(x,x_1)\,d\sigma(x)\,dr\\
&\quad+\frac{ds(2d-2-s)}{4(d-2)C_d r_0^{d}}
\int_0^{r_0}\!\int_{D_r(x_0)}rR_s(x,x_1)\,d\sigma(x)\,dr\\
&\quad+\frac{d}{C_d r_0^{d}}\int_{D_{r_0}(x_0)}\!
\Big(1-\frac{1}{4}|x-x_0|^2\Big)R_s(x,x_1)\,d\sigma(x),
\end{split}
\end{equation}
where $C_d=2{\pi^{d/2}}/{\Gamma\left(d/2\right)}$. 
\end{lemma}

\begin{proof}
Let $0<r<r_0$. Applying \eqref{lbop through laplace eq3} with $\Omega=D_r(x_0)$ and 
$f(x)=R_s(x,x_1)$ we get
\begin{equation}\label{lbop through laplace eq2 interm 7}
\begin{split}
C_d (d-2)R_s(x_0,x_1)&=
\int_{D_r(x_0)}R_{d-2}(x,x_0)
\,\Big(\!\!-\lbop + \frac{1}{4}\,(d-2)d\Big)R_s(x,x_1)\,d\sigma(x)\\
&\quad+\int_{\partial D_r(x_0)}\!\! 
R_{d-2}(x,x_0)\,\nabla R_s(x,x_1)\!\cdot\! \nu(x)\,d\sigma'(x)\\
&\quad-\int_{\partial D_r(x_0)}\!\! R_s(x,x_1)\,\nabla R_{d-2}(x,x_0)
\!\cdot\! \nu(x)\,d\sigma'(x)\\
&=:I_1(r)+I_2(r)+I_3(r),
\end{split}
\end{equation}
where, as before, $\sigma'$ stands for the $(d-1)$-dimensional Hausdorff measure.
The proof of \eqref{lbop through laplace form1} is based on multiplying \eqref{lbop through laplace eq2 interm 7} by $r^{d-1}$ and integrating over all $r\in(0,r_0)$. We deal with the three terms on the right hand side of \eqref{lbop through laplace eq2 interm 7} separately. On one hand, using \eqref{lbop through laplace eq1} we get
\begin{equation}\label{lbop through laplace eq2 interm 8}
\begin{split}
I_1(r)=(d-2-s)\int_{D_r(x_0)}R_{d-2}(x,x_0)
\,\Big(\frac{d-s}{4}\,R_s(x,x_1)+sR_{s+2}(x,x_1)\Big)\,d\sigma(x).
\end{split}
\end{equation}
Regarding $I_2(r)$, since $R_{d-2}(x,x_0)=r^{2-d}$ for all $x\in\partial D_r(x_0)$, the divergence theorem and \eqref{lbop through laplace eq1}
yield
\begin{equation}\label{lbop through laplace eq2 interm 9bis}
\begin{split}
I_2(r)&=r^{2-d}\int_{D_r(x_0)}\lbop R_s(x,x_1)\,d\sigma(x)\\
&=r^{2-d}s\int_{D_r(x_0)}
\Big(\frac{1}{4}\,(2d-2-s)R_s(x,x_1)-(d-2-s)R_{s+2}(x,x_1)\Big)\,d\sigma(x).
\end{split}
\end{equation}
Finally, arguing as in \eqref{lbop through laplace eq2 interm 6bis} we deduce that
\begin{equation}\label{lbop through laplace eq2 interm 6}
\begin{split}
I_3(r)&=\frac{d-2}{r^{d-1}}\int_{\partial D_r(x_0)}\!\!
\Big(\frac{1+(x_0\cdot x)}{2}\Big)^{1/2}\,R_s(x,x_1)\,d\sigma'(x).
\end{split}
\end{equation}
A combination of \eqref{lbop through laplace eq2 interm 6} and the smooth coarea formula, see \cite[page 160]{Chavel}, leads to
\begin{equation}\label{lbop through laplace eq2 interm 10}
\begin{split}
\int_0^{r_0} r^{d-1}I_3(r)\,dr
&=(d-2)\int_{D_{r_0}(x_0)}\frac{1+(x_0\cdot x)}{2}\,R_s(x,x_1)\,d\sigma(x)\\
&=(d-2)\int_{D_{r_0}(x_0)}\Big(1-\frac{|x-x_0|^2}{4}\Big)R_s(x,x_1)\,d\sigma(x).
\end{split}
\end{equation}
Therefore, if we multiply \eqref{lbop through laplace eq2 interm 7} by $r^{d-1}$ and we integrate over all $r\in(0,r_0)$, using \eqref{lbop through laplace eq2 interm 8},
\eqref{lbop through laplace eq2 interm 9bis} and 
\eqref{lbop through laplace eq2 interm 10} we finally get 
\eqref{lbop through laplace form1}.
\end{proof}


\begin{proposition}\label{prop lbop through laplace main estimate ds1}
Let $d>2$ and $0< s<d$. There exists $C>0$ only depending on $d$ and $s$ such that 
\begin{equation}\label{lbop through laplace main estimate ds1}
\bigg|R_s(x_0,x_1)-\dashint_{D_{r_0}(x_0)}\!R_s(x,x_1)\,d\sigma(x) \bigg|
\leq C\big(\psi_{d-2}(s)2^{2k}+1\big)2^{ks}r_0^{2}
\end{equation}
for all $k\geq0$, all $x_0,\,x_1\in\sph^{d}$ with $|x_0-x_1|\geq2^{-k}$ and all $0<r_0\leq2^{-k-2}$, where we have set $\psi_{d-2}(s)=0$ if $s=d-2$ and 
$\psi_{d-2}(s)=1$ otherwise.
\end{proposition}

\begin{proof}
Thanks to \eqref{lbop through laplace form1}, we can write
\begin{equation}\label{lbop through laplace main estimate ds1 aux}
\begin{split}
R_s(x_0,x_1)&-\frac{1}{|D_{r_0}(x_0)|}\int_{D_{r_0}(x_0)}\!R_s(x,x_1)\,d\sigma(x)\\
&=\frac{ds(d-2-s)}{(d-2)C_d r_0^{d}}
\int_0^{r_0}\!\int_{D_r(x_0)}\!\!r^{d-1}
\big(R_{d-2}(x,x_0)-r^{2-d}\big)R_{s+2}(x,x_1)\,d\sigma(x)\,dr\\
&\quad+\frac{d(d-s)(d-2-s)}{4(d-2)C_d r_0^{d}}
\int_0^{r_0}\!\int_{D_r(x_0)}r^{d-1}R_{d-2}(x,x_0)R_s(x,x_1)\,d\sigma(x)\,dr\\
&\quad+\frac{ds(2d-2-s)}{4(d-2)C_d r_0^{d}}
\int_0^{r_0}\!\int_{D_r(x_0)}rR_s(x,x_1)\,d\sigma(x)\,dr\\
&\quad-\frac{d}{4C_d r_0^{d}}\int_{D_{r_0}(x_0)}\!
|x-x_0|^2R_s(x,x_1)\,d\sigma(x)\\
&\quad+\Big(\frac{d}{C_d r_0^{d}}-\frac{1}{|D_{r_0}(x_0)|}\Big)\int_{D_{r_0}(x_0)}\!R_s(x,x_1)\,d\sigma(x)\\
&=:S_1+S_2+S_3+S_4+S_5.
\end{split}
\end{equation}
We are going to estimate the terms $S_1,\ldots,S_5$ separately. However, all the estimates rely basically on the assumptions $|x_0-x_1|\geq2^{-k}$ and $r_0\leq2^{-k-2}$. On one hand, we easily see that
\begin{equation}\label{lbop through laplace main estimate ds1 s1}
\begin{split}
|S_1|&\leq C\psi_{d-2}(s)2^{k(s+2)}r_0^{-d}
\int_0^{r_0}\!\int_{D_r(x_0)}\!\!r^{d-1}(|x-x_0|^{2-d}-r^{2-d})\,d\sigma(x)\,dr\\
&\leq C\psi_{d-2}(s)2^{k(s+2)}r_0^{2}.
\end{split}
\end{equation}
Similarly,
\begin{equation}\label{lbop through laplace main estimate ds1 s2s3s4}
|S_2|\leq C\psi_{d-2}(s)2^{ks}r_0^{2},
\qquad|S_3|\leq C2^{ks}r_0^{2}
\qquad\text{and}\qquad|S_4|\leq C2^{ks}r_0^{2}.
\end{equation}
Finally, by taking local chards in $\sph^d$, one can show that
$\big||D_{r}(x)|-C_d r^{d}/d\big|\leq C r^{d+2}$ for all $0\leq r\leq2$ and all $x\in\sph^d$. Hence,
\begin{equation}\label{lbop through laplace main estimate ds1 s5}
|S_5|\leq Cr_0^{-2d}\big||D_{r}(x)|-C_d r^{d}/d\big|
\int_{D_{r_0}(x_0)}\!R_s(x,x_1)\,d\sigma(x)
\leq C2^{ks}r_0^2.
\end{equation}
Plugging \eqref{lbop through laplace main estimate ds1 s1}, 
\eqref{lbop through laplace main estimate ds1 s2s3s4} and 
\eqref{lbop through laplace main estimate ds1 s5} in 
\eqref{lbop through laplace main estimate ds1 aux}, we obtain
\eqref{lbop through laplace main estimate ds1}, as desired.
\end{proof}


\begin{theorem}								\label{teo_estimate_difference}
Let $d>2$, $0< s<d$, and $\rho>0$. There exists $C>0$ only depending on $d$, $s$ and $\rho$ such that 
\begin{equation}\label{estimate double sum}
\begin{split}
\frac{1}{N^2}\sum_{i\neq j}\bigg|R_s(x_i,x_j)
-\dashint_{D_j}\dashint_{D_i}\!R_s(x,y)\,d\sigma(x)\,d\sigma(y)\bigg|
\leq C\epsilon^2\big(N^{-\frac{2}{d}}+N^{-1+\frac{s}{d}}\big),
\end{split}
\end{equation}
for all $0<\epsilon\leq\rho/8$, all $N\in\N$ and every sequence of points 
$\{x_j\}_{j=1,\ldots,N}\subset\sph^d$ such that
$|x_i-x_j|\geq \rho N^{-1/d}$ for all $i\neq j$, where we have set $D_j=D_{\epsilon N^{-1/d}}(x_j)$ for $j=1,\ldots,N$.
\end{theorem}


\begin{proof}
First of all, note that $D_i\cap D_j=\emptyset$ for all $i\neq j$ since 
$0<\epsilon\leq\rho/8$ and $|x_i-x_j|\geq \rho N^{-1/d}$. Therefore, the left hand side of \eqref{estimate double sum} is well defined and finite for all $s\geq0$. 

Given $i\neq j$, since $\rho N^{-1/d}\leq|x_i-x_j|\leq2$, there exists some integer 
$k\geq0$ such that 
\begin{equation}\label{estimate double sum pointwise number 2}
\rho N^{-1/d}/2\leq 2^{-k}\leq|x_i-x_j|\leq2^{-k+1}.
\end{equation} 
Then, using the triangle inequality and \eqref{lbop through laplace main estimate ds1}
we can estimate
\begin{equation}\label{estimate double sum pointwise ij}
\begin{split}
\bigg|R_s(x_i,x_j)-\dashint_{D_j}\dashint_{D_i}\!R_s&(x,y)\,d\sigma(x)\,d\sigma(y)\bigg|\\
&\leq\bigg|R_s(x_i,x_j)-\dashint_{D_j}R_s(x_i,y)\,d\sigma(y)\bigg|\\
&\quad+\dashint_{D_j}\bigg|R_s(x_i,y)-\dashint_{D_i}\!R_s(x,y)
\,d\sigma(x)\bigg|\,d\sigma(y)\\
&\leq C\big(\psi_{d-2}(s)2^{2k}+1\big)2^{ks}\epsilon^2 N^{-2/d}.
\end{split}
\end{equation}
In addition, due to the constraint $|x_i-x_j|\geq \rho N^{-1/d}$ for all $i\neq j$, it is not hard to show that there exists $C>0$ only depending on $d$ such that, for every $i$ and $k$,
\begin{equation}\label{estimate double sum pointwise number}
\#\{j:\,2^{-k}\leq|x_i-x_j|\leq2^{-k+1}\}\leq C2^{-kd}\rho^{-d}N.
\end{equation} 
Therefore, a combination of \eqref{estimate double sum pointwise number 2}, 
\eqref{estimate double sum pointwise ij} and 
\eqref{estimate double sum pointwise number} leads to
\begin{equation}\label{estimate double sum almost there}
\begin{split}
\sum_{i\neq j}\bigg|&R_s(x_i,x_j)
-\dashint_{D_j}\dashint_{D_i}\!R_s(x,y)\,d\sigma(x)\,d\sigma(y)\bigg|\\
&\leq C\sum_{1\leq i\leq N}\,\sum_{0\leq k\leq \log_2(\frac{2N^{1/d}}{\rho})}
\,\sum_{j:\,2^{-k}\leq|x_i-x_j|\leq2^{-k+1}}
\big(\psi_{d-2}(s)2^{2k}+1\big)2^{ks}\epsilon^2 N^{-2/d}\\
&\leq C\rho^{-d}\epsilon^2 N^{2-2/d}\sum_{0\leq k\leq \log_2(\frac{2N^{1/d}}{\rho})}
2^{-k(d-s)}\big(\psi_{d-2}(s)2^{2k}+1\big).
\end{split}
\end{equation}
Recall that if $s=d-2$ then $\psi_{d-2}(s)=0$, thus from 
\eqref{estimate double sum almost there} we obtain in this case
\begin{equation}\label{estimate double sum almost there 1}
\begin{split}
\sum_{i\neq j}\bigg|&R_{d-2}(x_i,x_j)
-\dashint_{D_j}\dashint_{D_i}\!R_{d-2}(x,y)\,d\sigma(x)\,d\sigma(y)\bigg|
\leq C\epsilon^2 N^{2-2/d}
\end{split}
\end{equation}
for some $C>0$ only depending on $d$, $s$ and $\rho$. On the other hand, if 
$s\neq d-2$ then $\psi_{d-2}(s)=1$, and from 
\eqref{estimate double sum almost there} we get
\begin{equation}\label{estimate double sum almost there 2}
\begin{split}
\sum_{i\neq j}\bigg|R_s(x_i,x_j)&
-\dashint_{D_j}\dashint_{D_i}\!R_s(x,y)\,d\sigma(x)\,d\sigma(y)\bigg|\\
&\leq C\rho^{-d}\epsilon^2 N^{2-2/d}\sum_{0\leq k\leq \log_2(\frac{2N^{1/d}}{\rho})}
2^{-k(d-2-s)}\\
&\leq C\epsilon^2 N^{2-2/d}(1+N^{(s-(d-2))/d})
= C\epsilon^2(N^{2-2/d}+N^{1+s/d})
\end{split}
\end{equation}
for some $C>0$ only depending on $d$, $s$ and $\rho$, as before. In any case, \eqref{estimate double sum} follows directly from 
\eqref{estimate double sum almost there 1} and 
\eqref{estimate double sum almost there 2}.
\end{proof}

\begin{remark}{\em
The estimate \eqref{estimate double sum} may not seem sharp but, as far as one bases it on a pointwise estimate of the factor inside the sum independently of $i$, in the spirit of \eqref{estimate double sum almost there}, one cannot expect anything better than \eqref{estimate double sum}. This is essentially because the estimate in \eqref{lbop through laplace main estimate ds1} is sharp for points in the sphere satisfying $2^{-k}\leq|x_i-x_j|\leq2^{-k+1}$. That is to say, there exists $C>0$ such that 
\begin{equation}\label{prop lbop through laplace main estimate ds1 remark}
\begin{split}
C^{-1}2^{k(s+2)}r_0^{2}
\leq\bigg|R_s(x_0,x_1)-\dashint_{D_{r_0}(x_0)}\!R_s(x,x_1)\,d\sigma(x)\bigg|
\leq C2^{k(s+2)}r_0^{2}
\end{split}
\end{equation}
whenever $s\neq d-2$ and $2^{-k}\leq|x_0-x_1|\leq2^{-k+1}$, where $x_0$, $x_1$ and $r_0$ are as in 
Proposition \ref{prop lbop through laplace main estimate ds1}. To see this simply note that $|S_1|$ in \eqref{lbop through laplace main estimate ds1 aux}, which is comparable to $2^{k(s+2)}r_0^{2}$ and has a positive integrand, is the dominant term in the decomposition given in \eqref{lbop through laplace main estimate ds1 aux} as $k\to+\infty$. Thus, all the other terms $S_j$ can be absorbed by $S_1$ for $k$ big enough, and everything is comparable for $k$ small. This reasoning gives the lower bound in \eqref{prop lbop through laplace main estimate ds1 remark}.
}\end{remark}

\begin{remark}\label{rmk separation thm asimpt}{\em
It is not hard to extend Theorem \ref{teo_estimate_difference} to the more general case $d\geq2$ and $0\leq s<d$ by a suitable modification of Lemma \ref{lm 32} and Proposition \ref{prop lbop through laplace main estimate ds1} using the corresponding identities from Lemma \ref{laplacian_of_riesz}, for example \eqref{log laplace eq3} instead of \eqref{lbop through laplace eq3} when $d=2$, or \eqref{log laplace eq1} instead of \eqref{lbop through laplace eq1} when $s=0$. We omit the details for the sake of shortness.
}\end{remark}


\begin{corollary}				\label{corollary_bounds_on_means}
Given $0\le s<d$, let $\{ x_i \}_{i=1}^N$ be an $N$ point set of minimizers of the Riesz $s$-energy. Then, there exist 
$\epsilon_0=\epsilon_0(s,d)>0$ such that, if $0<\epsilon<\epsilon_0$
and $D_j=D_{\epsilon N^{-1/d}}(x_j)$ for $j=1,\ldots,N$, 
\begin{equation}					\label{estimate-minimal-energy}
\begin{split}
 \frac{1}{N^2}\sum_{i\neq j} \dashint_{D_j}\dashint_{D_i}\!R_s(x,y)\,
d\sigma(x)\,d\sigma(y)
\leq  \frac{\mathcal{E}_s(N)}{N^2}+ C \epsilon^2 \big(N^{-\frac{2}{d}}+N^{-1+\frac{s}{d}}\big),
\end{split}
\end{equation}
for some constant $C>0$ depending only on $d$ and $s$.
\end{corollary}

\proof
If $d\geq2$ and $d-2\leq s<d$ the minimizers are separated $|x_i-x_j|\ge c N^{-1/d}$, 
and by taking $\epsilon_0>0$ small enough according to the separation the result follows from Theorem \ref{teo_estimate_difference} and Remark \ref{rmk separation thm asimpt}. 
For $d>2$ and $0\leq s<d-2$ we apply Lemma \ref{lemma_mean} to spherical caps of radius $\epsilon N^{-1/d}$
for $\epsilon>0$ small enough.
\qed


\section{Estimates of the Sobolev discrepancy} \label{s4}

In this section we derive the estimates of the Sobolev discrepancy stated in Theorem \ref{teo_sobolev_discrepancy}, which are sharp for the range $d-2\leq s<d$. 
The first result, that can be 
seen as a sort of Stolarky's invariance principle, generalizes Wolff's result on the Sobolev discrepancy for $\sph^2$ and $s=0$.

\begin{lemma}						\label{lemma:decomposition}
Let $D_1,\dots, D_N$ be spherical caps in $\sph^d$ of the same radius $r>0.$ Consider the measures
$$\mu_i=\frac{\chi_{D_i}}{\sigma(D_i)} \,\sigma,\qquad\mu=\frac{1}{N}\sum_{i=1}^N \mu_i -\frac{\sigma}{\omega_{d}}.$$ 
In particular, $\mu(\sph^d)=0.$ Let $K(x,y)=K(|x-y|)$ be a rotation invariant integrable kernel. Then,
\begin{equation}
\begin{split}
\frac{1}{N^2}  \sum_{i\neq j} \int_{\sph^d}\!\int_{\sph^d} K(x,y)\,d\mu_i(x)\,d\mu_j(y)
&=\frac{1}{\omega_d^2}\int_{\sph^d}\!\int_{\sph^d} K(x,y)\,d\sigma(x)\,d\sigma(y)
\\
& \quad+
\int_{\sph^d}\!\int_{\sph^d} K(x,y)\,d\mu(x)\,d\mu(y)\\
& \quad-\frac{1}{N}\dashint_{D}\dashint_{D} K(x,y)\,d\sigma(x)\,d\sigma(y),
\end{split}
\end{equation}
where $D$ is a spherical cap of radius $r$ centered at the north pole. 
\end{lemma}


\proof
Writing the measure $\mu$ in terms of its summands we have
\begin{equation}
\begin{split}
 \int_{\sph^d}\!\int_{\sph^d}   K(x,y)\,d\mu (x)\,d\mu (y)
 &=\frac{1}{N^2} \sum_{i,j} \int_{\sph^d}\!\int_{\sph^d} K(x,y)\,d\mu_i(x)\,d\mu_j(y) 
\\
&\quad
-\frac{2}{N{\omega_d}} \sum_{i} \int_{\sph^d}\!\int_{\sph^d} K(x,y)\,d\mu_i(x)\, d\sigma(y)\\ 
&\quad+
\frac{1}{\omega_d^2}\int_{\sph^d}\!\int_{\sph^d} K(x,y)\,d\sigma(x)\,d\sigma(y).
\end{split}
\end{equation}
Now, observing that
$$2\Big(\frac{1}{N}\sum_i d\mu_i(x)\Big)\frac{d\sigma(y)}{\omega_d}-\frac{d\sigma(x)\,d\sigma(y)}{\omega_d^2}=
\frac{d\sigma(x)\,d\sigma(y)}{\omega_d^2}+
2\,\frac{d\mu(x)\,d\sigma(y)}{\omega_d},$$
we obtain
\begin{equation}
\begin{split}
 \int_{\sph^d}\!\int_{\sph^d}  K(x,y)\,d\mu (x)\,d\mu (y)
 &=\frac{1}{N^2} \sum_{i\neq j} \int_{\sph^d}\!\int_{\sph^d} K(x,y)\,d\mu_i(x)\,d\mu_j(y) \\
&\quad+\frac{1}{N^2} \sum_{i} \int_{\sph^d}\!\int_{\sph^d} K(x,y)\,d\mu_i(x)\,d\mu_i(y)\\ 
&\quad -\frac{1}{\omega_d^2}\int_{\sph^d}\!\int_{\sph^d} K(x,y)\,d\sigma(x)\,d\sigma(y)\\
&\quad-\frac{2}{\omega_d}\int_{\sph^d}\!\int_{\sph^d} K(x,y)\,d\mu(x)\,d\sigma(y).
\end{split}
\end{equation}
The last integral vanishes because $\mu(\sph^d)=0$ and $K(x,y)$ is rotation invariant. Moreover, by rotation invariance, the integrals
$$\int_{\sph^d}\!\int_{\sph^d} K(x,y)\,d\mu_i(x)\,d\mu_i(y)$$
are all equal and independent of the center of the spherical cap. 
\qed


\begin{proposition}				\label{prop_decomposition}
Let $D_1,\dots, D_N$ be spherical caps in $\sph^d$ of the same radius $\epsilon N^{-1/d}.$ Consider the measures
$$\mu_i=\frac{\chi_{D_i}}{\sigma(D_i)} \,\sigma,\qquad\mu=\frac{1}{N}\sum_{i=1}^N \mu_i -\frac{\sigma}{\omega_{d}}.$$ 
Then, 
\begin{equation}
 E_s(\widetilde{\sigma})+E_s(\mu)-\frac{1}{N^2} 
\sum_{i\neq j}  \int_{\sph^d}\!\int_{\sph^d} R_s(x,y) \,d\mu_i(x)\,d\mu_j(y)\approx \epsilon^{-s} N^{-1+\frac{s}{d}},
\end{equation}
for $0<s<d,$ where $\widetilde{\sigma}=\sigma/\omega_d$ is the normalized surface measure in $\sph^d$. If $s=0$, then
\begin{equation}
\begin{split}
 \frac{1}{N^2} \sum_{i\neq j}  \int_{\sph^d}\!\int_{\sph^d} R_0(x,y) & \,d\mu_i(x)\,d\mu_j(y)
=E_0(\widetilde{\sigma})+E_0(\mu)-\frac{1}{d}\frac{\log N}{N}+O(N^{-1}).
\end{split}
\end{equation}
\end{proposition}

\proof
For $0<s<d$ we take $K=R_s$ in Lemma \ref{lemma:decomposition} to deduce that
$$ \frac{1}{N^2} \sum_{i\neq j} \! \int_{\sph^d}\int_{\sph^d} R_s(x,y) \,d\mu_i(x)\,d\mu_j(y)
=  {E_s(\widetilde\sigma)}+E_s(\mu)-\frac{1}{N}\dashint_{D}\dashint_{D} R_s(x,y)\,d\sigma(x)\,d\sigma(y),
$$
where $D$ denotes a spherical cap of radius $\epsilon N^{-1/d}$ centered at the north pole $n=(0,\dots, 0,1).$
To estimate this last integral 
we use normal coordinates around
the north pole, say,
$$
\Phi(x)=
\begin{cases}
(0,\dots , 0,1), &\text{if } x=0,\\
\big(\frac{x}{|x|}\sin |x|,\cos |x|\big), &\text{if } x\ne 0,
\end{cases}
$$
for $x\in \R^d$ with $|x|\le \pi.$
For every $r>0$, we have
\begin{equation}
\begin{split}
&\int_{D(n,r) }\!\int_{D(n,r) }R_s (x,y)\,d\sigma(x)\,d\sigma(y)\\
&\quad=\int_{B\big(0,\,\arccos(1-\frac{r^2}{2})\big) } \!\int_{B\big(0,\,\arccos(1-\frac{r^2}{2})\big) } \!\!R_s (\Phi(x),\Phi(y))\Big(\frac{\sin |x|}{|x|}\Big)^{d-1} \Big(\frac{\sin |y|}{|y|}\Big)^{d-1} dx\, dy,
\end{split}
\end{equation}
where $B(a,r)$ is the ball in $\R^d$ of center $a\in \R^d$ and radius $r>0.$
As there exist constants $C,\,c>0$ such that $c|x-y|\le |\Phi(x)-\Phi(y)|\le C|x-y|$ 
and $1/2\le \sin t/t\le  1$ for all $|t|\le \pi/2$, we deduce that 
$$\int_{D(n,r) }\!\int_{D(n,r) }R_s (x,y)\,d\sigma(x)\,d\sigma(y)\approx 
\int_{B\big(0,\,\arccos(1-\frac{r^2}{2})\big) }\! \int_{B\big(0,\,\arccos(1-\frac{r^2}{2})\big) } R_s (x,y)\,  dx\, dy.$$
Finally, it is easy to check that, for $0<s<d$,
$$\dashint_{B(0,r) }\dashint_{B(0,r) }R_s (x-y)\, dx\, dy=  r^{-s}  \dashint_{B(0,1) }\dashint_{B(0,1) }R_s (x-y)\, dx\, dy,$$
and
$$\dashint_{B(0,r) }\dashint_{B(0,r) }R_0 (x-y)\, dx\, dy= \log \frac{1}{r}+ \dashint_{B(0,1) }\dashint_{B(0,1) }R_s (x-y)\, dx\, dy,$$
where we used that $|B(0,1)|=\omega_{d-1}/d.$ Then, since $\frac{1}{2}|B(0,r)|\le \sigma(D(n,r))\le |B(0,r)|$, we conclude that, for $0<s<d$,
\begin{equation}			\label{low:estimate}
 C \frac{N^{s/d}}{\epsilon^s}\le \dashint_{D}\dashint_{D} R_s(x,y)\,d\sigma(x)\,d\sigma(y)\le C^{-1} \frac{N^{s/d}}{\epsilon^s}
\end{equation}
for some constant $C >0$ depending only on $s$ and $d.$ 
In the case $s=0$, we get that
$$\dashint_{D}\dashint_{D} R_0(x,y)\,d\sigma(x)\,d\sigma(y)=\frac{1}{d} \log N +C+o(1),\;\;N \to +\infty,$$
where $C\in \R$ depends only on $d.$
\qed

\proof[Proof of Theorem \ref{teo_sobolev_discrepancy}] 
We first show that the Sobolev discrepancy of the measures associated to any set of $N$ points is bounded below. 
  Given $X_N=\{x_1,\dots , x_N \}\subset \sph^d$ and $\epsilon>0$ define the measures
$$\mu_i=\frac{\chi_{D_i}}{\sigma(D_i)} \,\sigma,\qquad\mu_{X_N,\epsilon}=\frac{1}{N}\sum_{i=1}^N \mu_i -\frac{\sigma}{\omega_{d}},$$ 
where $D_j=D_{\epsilon N^{-1/d}}(x_j).$
Since the $\mu_i$ are probability measures,
we have

\begin{equation}    \label{second:low:estimate}
\begin{split}
 \mathcal{E}_s(N) & \le \int_{\sph^d}\dots \int_{\sph^d} \Big( \sum_{i\neq j}R_s(x_i,x_j)  \Big) \,d\mu_1(x_1)\dots d\mu_N(x_N)
\\
&
=
\sum_{i\neq j}\iint R_{s}(x,y)\,d\mu_i(x)\,d\mu_j(y).
\end{split}
\end{equation}
Combining (\ref{second:low:estimate}), the lower estimate for the minimal energy in (\ref{knownboundsenergy}), (\ref{knownboundsenergy2}), Proposition 
\ref{prop_decomposition}, and Lemma \ref{lemma_equivalence} we get the lower bound
\begin{equation}  \label{bound:name}
 D_{s,d}^\epsilon(X_N)^2\ge (-c+C\epsilon^{-s})N^{-1+\frac{s}{d}},
\end{equation}
where $c>0$ is the constant in (\ref{knownboundsenergy}) and $C>0$ is the constant in (\ref{low:estimate}).
Observe that the bound in (\ref{bound:name}) is not trivial  for $\epsilon$ small enough.

For the upper bound we use again Proposition \ref{prop_decomposition}, Corollary \ref{corollary_bounds_on_means}, the
upper estimates for the minimal energy in (\ref{knownboundsenergy}), (\ref{knownboundsenergy2}) and Lemma \ref{lemma_equivalence}. We obtain that, for all $0<\epsilon<\epsilon_0(s,d)$ and $0<s<d$,
$$D_{s,d}^\epsilon(X_N)^2\le C \epsilon^2 N^{-2/d}+(C+C\epsilon^{-s}+C\epsilon^2)N^{-1+\frac{s}{d}},$$
and for $s=0$ and $d>2$ 
$$D_{0,d}^\epsilon(X_N)^2\le (C \epsilon^2+C) N^{-2/d}.$$
\qed

\begin{remark}   \label{remark_historico}
{\em 
In the manuscript \cite{Wolff}, Wolff uses the asymptotic expansion of the discrete minimal energy 
\begin{equation}   				\label{new_bound_energy_log}
\sum_{i\ne j}\log \frac{1}{|x_i-x_j|}=\frac{N^2}{(4\pi)^2}\int_{\sph^2}\!\int_{\sph^2}
\log \frac{1}{|x-y|}\,d\sigma(x)\,d\sigma(y)-\frac{N}{2} \log N+O(N),
\end{equation}
as $N\to +\infty,$ which it seems it was not known at that time. In fact, in the manuscript Wolff mentions that he borrows the direction $\ge$ in (\ref{new_bound_energy_log}) from 
Elkies \cite[page 150]{Lan88} and proves the other, hence the  manuscript must precede 
Wagner's bound \cite{Wag92}. This agrees with the information we have from Eremenko about Wolff giving him the manuscript around 1992. 
We will schetch now the main ideas in the manuscript to prove the inequality $\le$ in (\ref{new_bound_energy_log}).
First, Wolff constructs area regular partitions on the sphere with pieces satisfying the Poincar\'e inequality. He calls
a set of points {\em allowable} if it is defined by taking one point on each of these pieces of the area regular partition.
Using Poincar\'e inequality he proves that allowable sets have minimal Sobolev discrepancy, i.e., of order $N^{-1}$. Finally, by using the case 
$s=0$ and $d=2$ of the decomposition 
in Lemma \ref{lemma:decomposition} and Lemma \ref{lemma_mean}, he proves that
allowable sets (and therefore Fekete points too) have logarithmic energy bounded above by the right hand side of (\ref{new_bound_energy_log}).
}\end{remark}


\section{From the Sobolev discrepancy to the spherical cap discrepancy}\label{s5}

The final ingredient to prove Theorem \ref{teo:main} is to estimate the spherical cap discrepancy using the bounds on the Sobolev discrepancy and a suitable test function. The 
main difficulty compared to Wolff's case \cite{Wolff}
is that we need the following result on interpolation to be able to estimate the test function.

\begin{lemma}					\label{lemma:interpolation}.
If $0\leq s<d$, there exists $C>0$ only depending on 
$d$ and $s$ such that
\begin{equation}\label{holder estimates 2}
\Big|\int fg\,d\sigma\Big|\leq C\|f\|^{\theta}_{\mathbb{H}^{[(d-s)/2]+1}(\sph^d)}
\|f\|^{1-\theta}_{\mathbb{H}^{[(d-s)/2]}(\sph^d)}  \|g \|_{\mathbb{H}^{\frac{s-d}{2}}(\sph^d)}
\end{equation}
for all $f\in \CC^\infty(\sph^d)$ and $g\in L^2(\sph^d)$, where 
$\theta=(d-s)/2-[(d-s)/2]\in[0,1)$.
\end{lemma}
\begin{proof}
We decompose $f$ and $g$ in spherical harmonics as 
\begin{equation}
f=\sum_{\ell,k}f_{\ell,k}\,Y_{\ell,k},\qquad
g=\sum_{\ell,k}g_{\ell,k}\,Y_{\ell,k},
\end{equation}
where
$f_{\ell,k}=\int f\,Y_{\ell,k}\,d\sigma$ and
$g_{\ell,k}=\int g\,Y_{\ell,k}\,d\sigma$. Then, using Cauchy-Schwarz inequality, \eqref{symbol Riesz3} and \eqref{symbol Riesz1}, we easily get
\begin{equation}\label{holder estimates eq1}
\begin{split}
\Big|\int fg\,d\sigma\Big|^2&\leq\Big(\sum_{\ell,k}|f_{\ell,k}||g_{\ell,k}|\Big)^2
\lesssim \sum_{\ell,k}(1+\ell^{d-s})|f_{\ell,k}|^2\sum_{\ell,k}A_{\ell,s}|g_{\ell,k}|^2\\
&= \|g \|_{\mathbb{H}^{\frac{s-d}{2}}(\sph^d)}^2   \sum_{\ell,k}(1+\ell^{d-s})|f_{\ell,k}|^2.
\end{split}
\end{equation}
For a general $s$, the last sum above may correspond to a Sobolev norm of noninteger order. We are going to estimate it, by interpolation, in terms of Sobolev norms of integer order. Let $m\in\N\cup\{0\}$ be such that $s\in[d-2(m+1),d-2m)$ and set $\delta=(d-s)/2-m$, hence $0<\delta\leq1$ and $d-s=2(m+\delta)$. 

Assume first that $\delta=1$, thus $d-s=2(m+1)$ and 
\begin{equation}
\sum_{\ell,k}(1+\ell^{d-s})|f_{\ell,k}|^2=\sum_{\ell,k}(1+\ell^{2(m+1)})|f_{\ell,k}|^2
\approx\|f\|^2_{\mathbb{H}^{m+1}(\sph^d) }=\|f\|^2_{\mathbb{H}^{(d-s)/2}(\sph^d)}.
\end{equation}
With this at hand, \eqref{holder estimates eq1} leads to \eqref{holder estimates 2} when $\theta=0$.

Assume now that $0<\delta<1$. Using that 
$\ell\geq0$, we can estimate 
\begin{equation}\label{holder estimates eq2}
1+\ell^{d-s}=1+\ell^{2(m+\delta)}\leq(1+\ell^{m})^2(1+\ell^{2\delta}).
\end{equation}
Since $1<1/\delta<+\infty$, a combination of \eqref{holder estimates eq2} with  
H\"older inequality yields
\begin{equation}
\begin{split}
\sum_{\ell,k}(1+\ell^{d-s})|f_{\ell,k}|^2
&\leq\sum_{\ell,k}\big((1+\ell^{m})|f_{\ell,k}|\big)^{2\delta}
(1+\ell^{2\delta})\big((1+\ell^{m})|f_{\ell,k}|\big)^{2-2\delta}\\
&\leq\Big(\sum_{\ell,k}\big((1+\ell^{m})|f_{\ell,k}|\big)^{2}
(1+\ell^{2\delta})^{\frac{1}{\delta}}\Big)^{\delta}
\Big(\sum_{\ell,k}\big((1+\ell^{m})|f_{\ell,k}|\big)^{2}
\Big)^{1-\delta}\\
&\lesssim\Big(\sum_{\ell,k}\big(1+\ell^{2(m+1)}\big)|f_{\ell,k}|^2\Big)^{\delta}
\Big(\sum_{\ell,k}\big(1+\ell^{2m}\big)|f_{\ell,k}|^2\Big)^{1-\delta}\\
&\approx\|f\|^{2\delta}_{\mathbb{H}^{m+1}(\sph^d)}\|f\|^{2-2\delta}_{\mathbb{H}^{m}(\sph^d)}. 
\end{split}
\end{equation}
The fact that $\delta<1$ leads to $m=[(d-s)/2]$, and therefore we conclude that
\begin{equation}
\Big(\sum_{\ell,k}(1+\ell^{d-s})|f_{\ell,k}|^2\Big)^{1/2}
\lesssim\|f\|^{(d-s)/2-[(d-s)/2]}_{\mathbb{H}^{[(d-s)/2]+1}(\sph^d)}
\|f\|^{1+[(d-s)/2]-(d-s)/2}_{\mathbb{H}^{[(d-s)/2]}(\sph^d)}.
\end{equation}
This, together with \eqref{holder estimates eq1}, proves \eqref{holder estimates 2} when $0<\theta<1$.
\end{proof}

Finally, the following proposition combined with Theorem \ref{teo_sobolev_discrepancy} proves Theorem \ref{teo:main}.

\begin{proposition}					\label{sobolev_SCdiscrepancy}
Given $0\leq s<d$, $\epsilon_0>0$ and $C_1>0$, there exists $C_2>0$ only depending on $d$, $s$, $\epsilon_0$, and $C_1$ 
such that, for every set
$X_N=\{x_1,\dots , x_N\} \subset\sph^d$ with Sobolev discrepancy
\begin{equation}\label{sph cap discr main assumption}
D_{s,d}^{\epsilon_0}(X_N)  \le C_1 \big(N^{-\frac{1}{d}}+N^{-\frac{1}{2}+\frac{s}{2d}}\big),
\end{equation}
the spherical cap discrepancy of $X_N$ satisfies
\begin{equation}\label{sph cap discr main estimate}  
\sup_D\, \Bigl| \frac{\# (X_N\cap D)}{N}- 
\frac{\sigma(D)}{\sigma(\sph^d)} \Bigr|
\leq C_2
\Big( \chi_{[0,d-2]}(s) N^{-\frac{2}{d(d-s+1)}}
+\chi_{(d-2,d)}(s) N^{-\frac{2(d-s)}{d(d-s+4)}} \Big),
\end{equation}
where the supremum in the left hand side of \eqref{sph cap discr main estimate} runs over all spherical caps $D\subset \sph^d.$ 
\end{proposition}


\begin{proof}
Let $D=D_r(z)$ for $z\in \sph^d$ and $r>0.$ Given $0<\epsilon<r/2$ let 
$f^\pm_\epsilon\in \CC^\infty(\sph^d)$ be such that $0\leq f_\epsilon^\pm\leq1$,
\begin{equation}\label{test functions defi}
f^+_\epsilon(x)=
\begin{cases}
1, &\text{if }|x-z|<r+\epsilon,\\
0, &\text{if }|x-z|>r+2\epsilon,
\end{cases}\qquad
f^-_\epsilon(x)=
\begin{cases}
1, &\text{if }|x-z|<r-2\epsilon,\\
0, &\text{if }|x-z|>r-\epsilon.
\end{cases}
\end{equation}
Note that 
\begin{equation}\label{sph cap discr main estimate sobolev test2}
\sigma(D)-C\epsilon\leq\int f_\epsilon^-\,d\sigma
\leq\int f_\epsilon^+\,d\sigma\leq \sigma(D)+C\epsilon
\end{equation}
for some constant $C>0$ only depending on $d$.
It is not hard to see that $f^\pm_\epsilon$ can be taken in such a way that
\begin{equation}\label{sph cap discr main estimate sobolev test}
\|f^\pm_\epsilon\|_{\mathbb{H}^{m}(\sph^d)}
\leq C\big(1+\epsilon^{-m+\frac{1}{2}}\big)\qquad\text{for all $m\in\N\cup\{0\}$.}
\end{equation}
For example, take a smooth function $\phi:\R\to\R$ such that 
$\chi_{(-\infty,0]}\leq\phi\leq \chi_{(-\infty,1)}$ and set 
$f^+_\epsilon(x)=\phi\big(\epsilon^{-1}(|x-z|-r-\epsilon)\big)$ and
$f^-_\epsilon(x)=\phi\big(\epsilon^{-1}(|x-z|-r+2\epsilon))\big)$. We leave the details of checking \eqref{sph cap discr main estimate sobolev test} for the reader.

Observe that if $\epsilon> \epsilon_0 N^{-1/d}$ then $f^+_\epsilon\equiv 1$ in 
$D_{\epsilon_0 N^{-1/d}}(x_j)$ for all $x_j\in D$.
Recall also form Definition \ref{defi sobolev discrep mu} that $\mu_{X_N,\epsilon_0}= h\sigma$ with 
\begin{equation}\label{sph cap discr main estimate sobolev test5}
h=\frac{1}{N}\sum_{j=1}^N \frac{\chi_{D_j}}{\sigma(D_j)}-\frac{1}{\sigma(\sph^d)},
\qquad\text{$D_j=D_{\epsilon_0 N^{-1/d}}(x_j)$ for all $j=1,\ldots,N$.}
\end{equation}
Therefore,
\begin{equation}\label{sph cap discr main estimate sobolev test3}
\begin{split}
\frac{\# \big(D \cap X_N \big)}{N}
&\leq\frac{1}{N}\sum_{j=1}^N\frac{1}{\sigma(D_j)}\int_{D_j}f_\epsilon^+\,d\sigma
=\int f_\epsilon^+\,d\mu_{X_N,\epsilon_0}+\frac{1}{\sigma(\sph^d)}\int f_\epsilon^+\,d\sigma.
\end{split}
\end{equation}
Similarly, since $f^-_\epsilon\leq 1$ and it is supported on the spherical cap or radius $r-\epsilon$ centered at $z$, we deduce that
\begin{equation}\label{sph cap discr main estimate sobolev test4}
\begin{split}
\frac{\# \big(D \cap X_N \big)}{N}
&\geq\frac{1}{N}\sum_{j=1}^N\frac{1}{\sigma(D_j)}\int_{D_j}f_\epsilon^-\,d\sigma
=\int f_\epsilon^-\,d\mu_{X_N,\epsilon_0}+\frac{1}{\sigma(\sph^d)}\int f_\epsilon^-\,d\sigma.
\end{split}
\end{equation} 

We have all the ingredients to prove \eqref{sph cap discr main estimate}. On one hand, if we first combine \eqref{sph cap discr main estimate sobolev test3} and
\eqref{sph cap discr main estimate sobolev test2}, and then we use 
that $\mu_{X_N,\epsilon_0}= h\sigma$, \eqref{holder estimates 2}, \eqref{sph cap discr main estimate sobolev test} and \eqref{sph cap discr main assumption},
we get
\begin{equation}
\begin{split}
\frac{\# \big(D \cap X_N \big)}{N} & -\frac{\sigma(D)}{\sigma(\sph^d)}
\leq\int f_\epsilon^+\,d\mu_{X_N,\epsilon_0}+C\epsilon
=\int f_\epsilon^+h\,d\sigma+C\epsilon\\
&\leq C\,\|f_\epsilon^+\|^{(d-s)/2-[(d-s)/2]}_{\mathbb{H}^{[(d-s)/2]+1}(\sph^d)  }
\|f_\epsilon^+\|^{1-(d-s)/2+[(d-s)/2]}_{ \mathbb{H}^{[(d-s)/2]}(\sph^d)   } \|h\|_{\mathbb{H}^{(s-d)/2}(\sph^d) }+C\epsilon\\
&\leq C\big(1+\epsilon^{-[\frac{d-s}{2}]-\frac{1}{2}}\big)^{\frac{d-s}{2}-[\frac{d-s}{2}]}
\big(1+\epsilon^{-[\frac{d-s}{2}]+\frac{1}{2}}\big)^{1-\frac{d-s}{2}+[\frac{d-s}{2}]}\\
&\hskip160pt\times\big(N^{-\frac{2}{d}}+N^{-1+\frac{s}{d}}\big)^{\frac{1}{2}}
+C\epsilon\\
&\leq C\,\frac{1+\epsilon^{-([\frac{d-s}{2}]-\frac{1}{2})(1-\frac{d-s}{2}+[\frac{d-s}{2}])}}
{\epsilon^{([\frac{d-s}{2}]+\frac{1}{2})(\frac{d-s}{2}-[\frac{d-s}{2}])}}
\big(N^{-\frac{1}{d}}+N^{-\frac{1}{2}+\frac{s}{2d}}\big)
+C\epsilon.
\end{split}
\end{equation}

On the other hand, combining \eqref{sph cap discr main estimate sobolev test4} and
\eqref{sph cap discr main estimate sobolev test2}, and then using
that $\mu_{X_N,\epsilon_0}= h\sigma$, \eqref{holder estimates 2}, \eqref{sph cap discr main estimate sobolev test} and \eqref{sph cap discr main assumption},
we obtain
\begin{equation}
\begin{split}
\frac{\# \big(D \cap X_N \big)}{N}  -\frac{\sigma(D)}{\sigma(\sph^d)}
&\geq\int f_\epsilon^-\,d\mu_{X_N,\epsilon_0}-C\epsilon
=\int f_\epsilon^-h\,d\sigma-C\epsilon\\
&\geq -C\,\frac{1+\epsilon^{-([\frac{d-s}{2}]-\frac{1}{2})(1-\frac{d-s}{2}+[\frac{d-s}{2}])}}
{\epsilon^{([\frac{d-s}{2}]+\frac{1}{2})(\frac{d-s}{2}-[\frac{d-s}{2}])}}
\big(N^{-\frac{1}{d}}+N^{-\frac{1}{2}+\frac{s}{2d}}\big)
-C\epsilon.
\end{split}
\end{equation}
In conclusion, we obtain the estimate
\begin{equation}\label{almost there}
\begin{split}
\Big|\frac{\# \big(D \cap X_N \big)}{N} & -\frac{\sigma(D)}{\sigma(\sph^d)} \Big|\\
&\leq C\,\frac{1+\epsilon^{-([\frac{d-s}{2}]-\frac{1}{2})(1-\frac{d-s}{2}+[\frac{d-s}{2}])}}
{\epsilon^{([\frac{d-s}{2}]+\frac{1}{2})(\frac{d-s}{2}-[\frac{d-s}{2}])}}
\big(N^{-\frac{1}{d}}+N^{-\frac{1}{2}+\frac{s}{2d}}\big)+C\epsilon.
\end{split}
\end{equation}
In order to deal with the right hand side of \eqref{almost there}, we consider two different cases: $d-2<s<d$ and $0\leq s\leq d-2$.

Assume first that $d-2<s<d$, thus $[(d-s)/2]=0$. Then, \eqref{almost there} leads to
\begin{equation}\label{almost there 1}
\begin{split}
\Big|\frac{\# \big(X_N \cap  D \big)}{N} & -\frac{\sigma(D)}{\sigma(\sph^d)} \Big|
\leq C\epsilon^{-\frac{d-s}{4}}N^{-\frac{d-s}{2d}}+C\epsilon.
\end{split}
\end{equation}
Remember that this estimate holds whenever $\epsilon>\epsilon_0 N^{-1/d}$, hence we can take
$$\epsilon=N^{-\frac{2(d-s)}{d(d-s+4)}}$$ for all $N$ big enough, and then 
\eqref{almost there 1} yields \eqref{sph cap discr main estimate}.

Let us deal now with the case $0\leq s\leq d-2$. From \eqref{almost there} we get
\begin{equation}\label{almost there 2}
\begin{split}
\Big|\frac{\# \big(X_N \cap  D \big)}{N} & -\frac{\sigma(D)}{\sigma(\sph^d)} \Big|
\leq C\epsilon^{-\frac{d-s-1}{2}}N^{-\frac{1}{d}}+C\epsilon.
\end{split}
\end{equation}
As before, this holds whenever $\epsilon>\epsilon_0 N^{-1/d}$, hence we can take $\epsilon=N^{-\frac{2}{d(d-s+1)}}$ for all $N$ big enough, and then 
\eqref{almost there 2} yields \eqref{sph cap discr main estimate}.
\end{proof}

\end{document}